\documentclass[12pt]{elsarticle}

\usepackage{amsmath,amssymb,amsthm}
\usepackage{mathrsfs}
\usepackage[bookmarks=true,
            bookmarksnumbered=true,
            bookmarksopen=true,
            colorlinks,
            pdfborder=001,
            linkcolor=black]{hyperref}
\usepackage{bm}
\usepackage{array}
\usepackage{float}
\usepackage{color}
\usepackage{lineno}
\usepackage{subcaption}
\usepackage{stmaryrd}
\usepackage{extarrows}
\usepackage{subcaption}
\usepackage{multirow}
\newtheorem{theorem}{Theorem}[section]
\newtheorem{lemma}[theorem]{Lemma}
\newdefinition{rmk}{Remark}[section]
\newdefinition{definition}{Definition}
\newtheorem{example}{Example}[section]

\newproof{pf}{Proof}
\numberwithin{equation}{section}
\numberwithin{figure}{section}
\numberwithin{table}{section}
\allowdisplaybreaks

\begin{document}
\begin{frontmatter}

  \title{A physical-constraints-preserving
genuinely multidimensional HLL scheme
for the special relativistic hydrodynamics}

  \author{Dan Ling}
  \ead{danling@xjtu.edu.cn}
  \address{School of Mathematics and Statistics, Xi'an Jiaotong University, Xi'an 710049, P.R. China}
  \author{Huazhong Tang\corref{cor1}}
  \ead{hztang@math.pku.edu.cn}
  \address{Center for Applied Physics and Technology, HEDPS and LMAM,
    School of Mathematical Sciences, Peking University, Beijing 100871, P.R. China}
  \cortext[cor1]{Corresponding author. Fax:~+86-10-62751801.}

  \begin{abstract}
This paper   develops the  genuinely multidimensional HLL Riemann solver
and finite volume scheme for the two-dimensional  special relativistic hydrodynamic 
equations on Cartesian meshes
and  studies its physical-constraint-preserving (PCP) property.
Several  numerical results demonstrate the accuracy, the performance and
 the resolution of the shock waves and the genuinely multi-dimensional wave structures of the proposed  PCP scheme.
  \end{abstract}


\end{frontmatter}
\section{Introduction}\label{intro}

The paper is concerned with the physical-constraints-preserving (PCP) genuinely multidimensional finite volume scheme for
the special relativistic hydrodynamics (RHD),
  which plays a major role in astrophysics, plasma physics and nuclear
physics etc.,
where the
fluid moves at extremely high velocities
near the speed of light so that the relativistic effects become important.
 In the (rest) laboratory frame, the two-dimensional (2D) special RHD equations governing an ideal fluid flow can be written in the divergence form
\begin{equation}\label{eq27}
\frac{\partial\bm{U}}{\partial t}+\sum\limits_{\ell=1}^2\frac{\partial\bm{F}_\ell(\bm{U})}
{\partial x_\ell}=0,
\end{equation}
where the conservative vector $\bm{U}$ and the flux $\bm{F}_\ell$ are defined respectively by
\begin{equation}\label{eq28}
\bm{U}=(D, \bm{m}, E)^T, \quad \bm{F}_\ell=\left(Du_\ell,
\bm{m}u_\ell+p\bm{e}_\ell,(E+p)u_\ell \right)^T,\ \ \ell=1,2,
\end{equation}
here $D=\rho\gamma$, $\bm{m}=Dh\gamma\bm{u}$, $E=Dh\gamma-p$ and $p$ are the mass, momentum and
total energy relative to the laboratory frame and the gas pressure, respectively,
$\bm{e}_\ell$ is the row vector denoting the $\ell$-th row of the unit matrix of size $2$, $\rho$ is the   rest-mass density, $\bm{u}=(u,v)$ is the fluid velocity vector,
  $\gamma=1/\sqrt{1-|\bm{u}|^2}$ is the Lorentz factor,
 $h=1+e+\frac{p}{\rho}$ is the specific enthalpy, and $e$ is the specific internal energy.
Note that  natural units (i.e., the speed of light $c = 1$) has been used.
The system \eqref{eq27} should be closed via the equation of state (EOS), which has a general form of $p=p(\rho,e)$.
The current discussion is restricted to
the perfect gas, whose EOS is formulated as
\begin{equation}\label{eq15}
p=(\Gamma-1)\rho e,
\end{equation}
with the adiabatic index $\Gamma\in(1, 2]$. Such restriction on $\Gamma$ is reasonable
under the compressibility assumptions, and $\Gamma$ is taken as 5/3 for
the mildly relativistic case and 4/3 for
the ultra-relativistic case. In this case, for $i=1,2$,
the Jacobian matrix $\bm{A}_i(\bm{U})=\partial\bm{F}_i/\partial\bm{U}$ of the system \eqref{eq27} has   $4$
real eigenvalues, which are ordered from smallest to biggest as follows
\begin{equation*}
\begin{aligned}
&\lambda_{i}^{(1)}(\bm{U})=\frac{u_i(1-c_s^2)-c_s\gamma^{-1}\sqrt{1-u_i^2-c_s^2(|\bm{u}|^2-u_i^2)}}
{1-c_s^2|\bm{u}|^2},\\
&\lambda_{i}^{(2)}(\bm{U})=\lambda_{i}^{(3)}(\bm{U})=u_i,\\
&\lambda_{i}^{(4)}(\bm{U})=\frac{u_i(1-c_s^2)+c_s\gamma^{-1}\sqrt{1-u_i^2-c_s^2(|\bm{u}|^2-u_i^2)}}
{1-c_s^2|\bm{u}|^2},
\end{aligned}
\end{equation*}
where
$c_s$ is the speed of sound   defined by
\begin{equation*}
c_s=\sqrt{{\Gamma p}/{(\rho h)}},
\end{equation*}
and satisfies the following inequality
\begin{equation*}
c_s^2=\frac{\Gamma p}{\rho h}=\frac{\Gamma p}{\rho+\frac{p}{\Gamma-1}+p}=\frac{(\Gamma-1)\Gamma p}{(\Gamma-1)\rho+\Gamma p}<\Gamma-1\le1.
\end{equation*}

Due to the relativistic effect, especially the appearance of the Lorentz factor,
the system \eqref{eq27} become more strongly nonlinear than the non-relativistic case,
which leads to that their analytic treatment
is extremely difficult and challenging, except in some special case, for
instance, the 1D Riemann problem or the isentropic problem \cite{marti1,pant,lora}.
Because there are no explicit expressions of the primitive variable vector
$\bm{V}=(\rho,\bm{u},p)^T$ and the flux vectors $\bm{F}_i$ in terms of $\bm{U}$,  the value of $\bm{V}$   cannot be explicitly recovered from the $\bm{U}$ and should iteratively solve  the nonlinear equation, e.g. the pressure equation
\begin{equation*}
E+p=D\gamma+\frac{\Gamma}{\Gamma-1}p\gamma^2,
\end{equation*}
with
$\gamma=\big(1-|\bm{m}|^2/(E+p)^2\big)^{-1/2}$.
%
%
Besides those,
there are some physical constraints, such as
$\rho>0, p>0$ and $E\ge D$, as well as that the velocity can not exceed the speed of light,
i.e. $|\bm{u}|<c=1$.
For the RHD problems with large Lorentz factor or low density or pressure, or strong discontinuity, it is easy to obtain the negative density or pressure, or the larger velocity than the speed of light in numerical computations,
so that the eigenvalues of the Jacobian matrix or the Lorentz factor may become imaginary, leading directly to the ill-posedness of the discrete problem. Consequently,
there is great necessity and significance to develop robust and accurate
 physical-constraints-preserving (PCP) numerical schemes  for \eqref{eq27}, whose  solutions can satisfy the intrinsic physical constraints, or belong to the admissible states set  \cite{wu2015}
\begin{equation*}
\mathcal{G}=\left\{\bm{U}=(D,\bm{m},E)^T\big|~\rho>0, p>0, |\bm{u}|<1\right\},
\end{equation*}
or equivalently
\begin{equation*}
\mathcal{G}=\left\{\bm{U}=(D,\bm{m},E)^T\big|~D>0,E-\sqrt{D^2+|\bm{m}|^2>0}\right\}.
\end{equation*}
Based on that, one can prove some useful properties of $\mathcal{G}$, see \cite{wu2015}.
Although the following second lemma  is  formally different from Lemma 2.3, its proof can be completed by using the latter.
A slightly different proof is also given in \ref{appendix}.

\begin{lemma}\label{lem1}
The admissible state set $\mathcal{G}$ is convex.
\end{lemma}
\begin{lemma}\label{lem2}
If assuming $\bm{U},\bm{U}_1,\bm{U}_2\in \mathcal{G}$, then:
\begin{enumerate}[(\romannumeral1)]
\item $\kappa\bm{U}\in\mathcal{G}$ for all $\kappa>0$.
\item $a_1\bm{U}_1+a_2\bm{U}_2\in\mathcal{G}$ for all $a_1,a_2>0$.
\item $\alpha\bm{U}-\bm{F}_i(\bm{U}),-\beta\bm{U}+\bm{F}_i(\bm{U})\in \mathcal{G}$ for
$\alpha\ge\lambda_i^{(4)}(\bm{U})$, $\beta\le\lambda_i^{(1)}(\bm{U})$ and $i=1,2$.
\end{enumerate}
\end{lemma}


The study of numerical methods for the RHDs may date back to
the finite difference code via artificial viscosity  for the spherically symmetric general RHD equations in the Lagrangian coordinate \cite{may1,may2} and  for multi-dimensional RHD equations in the Eulerian coordinate \cite{wilson}.
Since 1990s, the numerical study of the RHD began to attract considerable attention, see the early review articles \cite{marti2,marti2015,font},
and
various modern shock-capturing methods with an exact or approximate Riemann solver have been developed for the RHD equations.
Some examples are   the two-shock Riemann solver
\cite{colella}, the Roe Riemann solver \cite{roe}, the HLL Riemann solver \cite{harten}
and the HLLC Riemann solver \cite{toro} and so on.
Some other higher-order accurate methods have also been well studied in the literature, e.g. the ENO (essentially non-oscillatory) and weighted ENO   methods \cite{dolezal,zanna,tchekhovskoy},
the discontinuous Galerkin  method \cite{RezzollaDG2011,zhao,ZhaoTang-CiCP2017,ZhaoTang-JCP2017}, the adaptive moving mesh methods \cite{he1,he2,Duan-Tang2020RHD2},
and the direct Eulerian GRP schemes \cite{yang2011direct,yang2012direct,wu2014,wu2016,Yuan-Tang2020}.
Most of the above mentioned methods are built on
the 1D Riemann solver, which is used to solve the local Riemann problem
on the cell interface  by picking up flow variations that are orthogonal
to the interface of the mesh and then give the exact or approximate Riemann solution. 
For the multi-dimensional problems,
there are still confronted with
enormous risks that the 1D Riemann solvers may lose their computational
efficacy and efficiency to some content, because 
some flow features propagating transverse to the mesh boundary
might be discarded, 
see \cite{vanLeer1993} for more details.
Therefore, it is necessary to capture much more flow features and then incorporate genuinely
multidimensional (physical) information into numerical methods.

In the early 1990s, owing to a shift from the finite-volume approach to the flctuation approach, the state of the art in genuinely multi-dimensional upwind differencing has made dramatic advances.
A early review of multidimensional upwinding may be found in \cite{vanLeer1993}.
For the linearized Euler equations, a genuinely multidimensional
first-order finite volume scheme
was   constructed by computing the exact solution of the Riemann problem
for a linear hyperbolic equation obtained from the linearized Euler equation was studied \cite{abgrall}.
Up to now, there have been some further developments on the
multidimensional Riemann solvers and corresponding numerical schemes, including
the multidimensional HLL schemes for solving Euler equations on unstructured triangular meshes
\cite{capdeville1,capdeville2}, the genuinely multidimensional HLL-type scheme
with convective pressure flux split Riemann solver \cite{mandal}, the multidimensional HLLE schemes
for gas dynamics \cite{wendroff,barsara},
the multidimensional HLLC schemes \cite{barsara2,barsara3} for hydrodynamics
and magnetohydrodynamics, and the genuinely two-dimensional
scheme for compressible flows in curvilinear coordinates \cite{qu} etc.
It is worth mentioning that all the aforementioned multidimensional schemes
are only for the non-relativistic fluid flows.
For the 2D special RHDs, a genuinely multidimensional scheme, the finite volume local evolution Galerkin method, was developed in \cite{wu2014b}.
%

On the other hand, based on the properties of $\mathcal{G}$, some PCP schemes were well developed for the special RHDs. They are the high-order accurate PCP finite difference WENO schemes,  discontinuous Galerkin (DG) methods and   Lagrangian  finite volume schemes proposed
in \cite{wu2015,wu2017,qin,wu2017a,ling}.
Such works were successfully extended to the special relativistic magnetohydrodynamics (RMHD) in \cite{wu2017m3as,wu2018zamp},
where the importance of divergence-free fields in achieving PCP methods is shown.
Recently, the entropy-stable schemes were also developed for the special RHD or RMHD
equations \cite{Duan-Tang2020RHD,Duan-Tang2020RMHD,Duan-Tang2020RHD2}.

This paper  will develop the  PCP 
genuinely multidimensional   finite volume scheme for the RHD equations \eqref{eq27}.
It is organized as follows. Section \ref{multi-hlle} derives
the 2D HLL Riemann solver for \eqref{eq27} and
studies the PCP property of its intermediate state.
Section \ref{scheme} presents 
 the PCP genuinely multidimensional  HLL scheme.
Section \ref{num}
conducts several numerical experiments to demonstrate the accuracy and good performance
of the present scheme. Section \ref{con} concludes the paper with some remarks.

\section{2D HLL Riemann solver}
\label{multi-hlle}
This section develops the genuinely multidimensional HLL Riemann solver for the 2D special RHD equations \eqref{eq27}  and the EOS
\eqref{eq15} on Cartesian meshes following Balsara's strategy \cite{barsara} and studies its PCP property.

\begin{figure}[h]
\centering
\vspace{-2ex}
\includegraphics[width=3.6in]{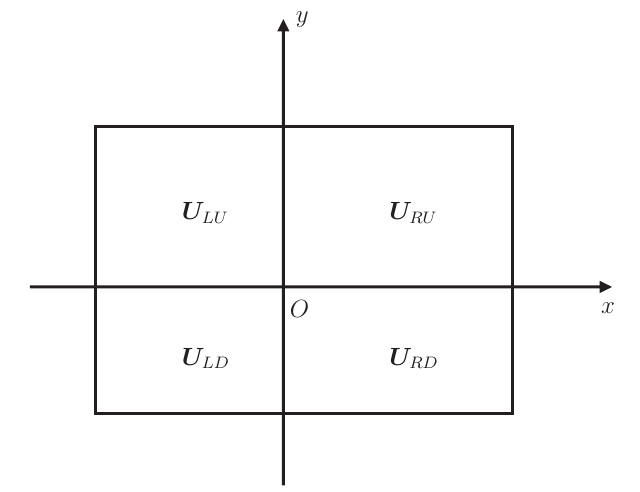}
\caption{The  initial data of
 the 2D Riemann problem at the origin $O$.}
\label{fig1}
\end{figure}

Consider the 2D Riemann problem of \eqref{eq27}
with the initial data as displayed in Fig. \ref{fig1}, where 
four constant states, $\bm{U}_{RU}$ (right-up), $\bm{U}_{LU}$ (left-up),
$\bm{U}_{LD}$ (left-down) and $\bm{U}_{RD}$ (right-down) are specified
in the first, second, third and fourth quadrants,  respectively, and $O$ is the coordinate origin.
Denote the largest left-, right-, up- and down-moving  speeds of
the elementary waves emerging from the initial discontinuities by $S_L$, $S_R$, $S_U$, $S_D$, respectively.
For instance, $S_L$ and $S_R$ are given obtained as follows:
calculate the largest left- and right-moving wave speeds in the 1D HLL solvers \cite{toro} for the two $x$-directional 1D
Riemann problems, denoted by $RP{\{\bm{U}_{LU},\bm{U}_{RU}\}}$
and $RP{\{\bm{U}_{LD},\bm{U}_{RD}\}}$,
and then 
minimize
two left-moving speeds of those 1D HLL solvers and maximize corresponding two right-moving speeds to give respectively $S_L$ and $S_R$.
Similarly, $S_D$ and $S_U$ are obtained by considering
two $y$-directional 1D Riemann problems denoted by $RP{\{\bm{U}_{LU},\bm{U}_{LD}\}}$
and $RP{\{\bm{U}_{RU},\bm{U}_{RD}\}}$.

In the following, the symbols  $(\bm{F}_1, \bm{F}_2)$, $(u_1,u_2)$, and $(x_1,x_2)$  will be replaced with $(\bm{F}, \bm{G})$, $(u,v)$, and $(x,y)$, respectively, and we only discuss the  case of that $S_L<0<S_R$ and $S_D<0<S_U$, because in the case of that $S_L$ and $S_R$ (or $S_D$ and $S_U$) have the same sign, our genuinely multidimensional scheme will only call the 1D Riemann solver.
Similar to the case of the 1D HLL Riemann solver,
 we have to derive the intermediate state
$\bm{U}^{\ast}$ in the approximate solution of
the above Riemann problem and corresponding fluxes $\bm{F}^\ast$ and $\bm{G}^\ast$.
For any given time $T>0$, choose
a three-dimensional cuboid $\mathbb V_{LRDU}(0,T)$ in the $(x,y,t)$ space as follows:
the top and bottom of the cuboid  are rectangles with
four vertices
$$
(TS_L,TS_D,0),(TS_R,TS_D,0),(TS_L,TS_U,0),(TS_R,TS_U,0),
$$
and
$$
(TS_L,TS_D,T),(TS_R,TS_D,T),(TS_L,TS_U,T),(TS_R,TS_U,T),
$$respectively.
Integrating \eqref{eq27} over such cuboid gives
\begin{align}\nonumber
\bm{U}^{\ast}\mathscr{A}&-\int_{TS_L}^{TS_R}\int_{TS_D}^{TS_U}
\bm{U}(x,y,0)dydx
\\ \nonumber
&+\int_{0}^{T}\int_{TS_D}^{TS_U}\bm{F}(\bm{U}(TS_R,y,0))dydt
-\int_{0}^{T}\int_{TS_D}^{TS_U}\bm{F}(\bm{U}(TS_L,y,0))dydt
\\
&+\int_{0}^{T}\int_{TS_L}^{TS_R}\bm{G}(\bm{U}(x,TS_U,0))dxdt
-\int_{0}^{T}\int_{TS_L}^{TS_R}\bm{G}(\bm{U}(x,TS_D,0))dxdt=0,
\label{eq3}
\end{align}
where
$$\bm{U}^{\ast}:=\frac{1}{\mathscr{A}}\int_{TS_D}^{TS_U}\int_{TS_L}^{TS_R}
\bm{U}(x,y,T)dxdy,\ \
\mathscr{A}:=T^2(S_R-S_L)(S_U-S_D).
$$
From \eqref{eq3}, one has
\begin{align}\nonumber
\bm{U}^{\ast}:&=\frac{1}{\mathscr{A}}\int_{TS_D}^{TS_U}\int_{TS_L}^{TS_R}
\bm{U}(x,y,T)dxdy
=\frac{S_RS_U\bm{U}_{RU}+S_LS_D\bm{U}_{LD}-S_RS_D\bm{U}_{RD}
-S_LS_U\bm{U}_{LU}}{(S_R-S_L)(S_U-S_D)}
\\ 
& -\frac{S_U(\bm{F}_{RU}-\bm{F}_{LU})-S_D(\bm{F}_{RD}-\bm{F}_{LD})}{(S_R-S_L)(S_U-S_D)}
-\frac{S_R(\bm{G}_{RU}-\bm{G}_{RD})-S_L(\bm{G}_{LU}-\bm{G}_{LD})}{(S_R-S_L)(S_U-S_D)}.
\label{eq43}
\end{align}
It is shown that the calculation of the  intermediate  state $\bm{U}^{\ast}$ requires four statuses
$\bm{U}_{LD}$, $\bm{U}_{LU}$, $\bm{U}_{RD}$, and $\bm{U}_{RU}$,
in other words, $\bm{U}^{\ast}$ contains the genuinely
multidimensional information.
In the special case of that
\begin{equation*}
\bm{U}_{LD}=\bm{U}_{LU},~~~\bm{U}_{RD}=\bm{U}_{RU},
\end{equation*}
 one has
\begin{equation*}
\bm{F}_{LD}=\bm{F}_{LU},~~~\bm{F}_{RD}=\bm{F}_{RU},~~~\bm{G}_{LD}=\bm{G}_{LU},
~~~\bm{G}_{RD}=\bm{G}_{RU},
\end{equation*}
and  
\begin{equation}\label{eq38}
\bm{U}^{\ast}=\frac{S_R\bm{U}_{RD}-S_L\bm{U}_{LU}+\bm{F}_{LD}-\bm{F}_{RD}}{S_R-S_L},
\end{equation}
which is the same as the  intermediate state in the 1D HLL Riemann solver in \cite{toro}.

Let us turn to obtain the resolved fluxes $\bm{F}^{\ast}$ and $\bm{G}^{\ast}$ for the multidimensional
Riemann solver
for the case of
$S_L<0<S_R$ and $S_D<0<S_U$.
Integrating respectively the system
\eqref{eq27} over the left portion and top portion
(or the right and bottom portions)
of the control
volume $\mathbb V_{LRDU}(0,T)$ yields
\begin{align} \nonumber
\int_{TS_L}^{0}\int_{TS_D}^{TS_U}\bm{U}(x,y,T)dydx&=\int_{TS_L}^{0}
\int_{TS_D}^{TS_U}\bm{U}(x,y,0)dydx
\\  \nonumber
~~~~~-\int_{0}^{T}&\int_{TS_D}^{TS_U}\big(\bm{F}(\bm{U}(0,y,t))-\bm{F}(\bm{U}(TS_L,y,t))\big)dydt\\
-\int_{0}^{T}&\int_{TS_L}^{0}\big(\bm{G}(\bm{U}(x,TS_U,t))-\bm{G}(\bm{U}(x,TS_D,t))\big)dxdt,
\label{eq11}
\\ \nonumber
\int_{TS_L}^{TS_R}\int_{0}^{TS_U}\bm{U}(x,y,T)dydx&=
\int_{TS_L}^{TS_R}\int_{0}^{TS_U}\bm{U}(x,y,0)dydx
\\ \nonumber
-\int_{0}^{T}\int_{0}^{TS_U}&\big(\bm{F}(\bm{U}(TS_R,y,t))
-\bm{F}(\bm{U}(TS_L,y,t))\big)dydt\\
-\int_{0}^{T}\int_{TS_L}^{TS_R}&\big(\bm{G}(\bm{U}(x,TS_U,t))
-\bm{G}(\bm{U}(x,0,t))\big)dxdt.
\label{eq12}
\end{align}
\begin{figure}[h]
\centering
\includegraphics[width=5.6in,height=2.0in]{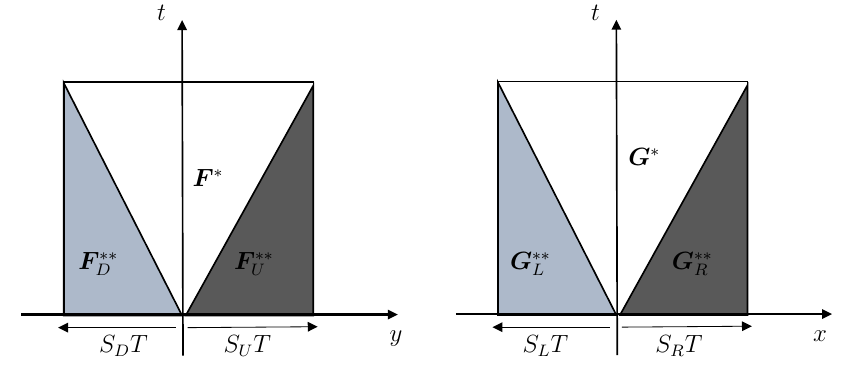}
\caption{Fluxes along the faces $x=0$ (left) and $y=0$ (right) respectively
consisting of several different portions.}
\label{facefluxes}
\end{figure}
 As shown in the Figure \ref{facefluxes},
the fluxes along the faces $x=0$ (left) and $y=0$ (right) respectively
are consisting of several different portions,
 so that
 the area integrals in \eqref{eq12}
on the $x=0$ and $y=0$ faces respectively read as
\begin{align}
\int_{0}^{T}\int_{TS_D}^{TS_U}\bm{F}(\bm{U}(0,y,t))&=\frac{T^2}{2}\bigg(S_U\bm{F}_U^{\ast\ast}
-S_D\bm{F}_D^{\ast\ast}+(S_U-S_D)\bm{F}^\ast\bigg),\label{add14}\\
\int_{0}^{T}\int_{TS_L}^{TS_R}\bm{G}(\bm{U}(x,0,t))&=\frac{T^2}{2}\bigg(S_R\bm{G}_R^{\ast\ast}
-S_L\bm{G}_L^{\ast\ast}+(S_G-S_L)\bm{G}^\ast\bigg),\label{add15}
\end{align}
where $\bm{F}_U^{\ast\ast},\bm{F}_D^{\ast\ast},\bm{G}_R^{\ast\ast}$,
and $\bm{G}_L^{\ast\ast}$ are corresponding 1D HLL fluxes in the
1D HLL Riemann solver and have the specific formulations of
\vspace{-1ex}\begin{align}
\bm{F}_U^{\ast\ast}&=\frac{1}{S_R-S_L}\bigg(S_R\bm{F}_{LU}-S_L\bm{F}_{RU}
+S_LS_R(\bm{U}_{RU}-\bm{U}_{LU})\bigg),\label{add16}\\
\bm{F}_D^{\ast\ast}&=\frac{1}{S_R-S_L}\bigg(S_R\bm{F}_{LD}-S_L\bm{F}_{RD}
+S_LS_R(\bm{U}_{RD}-\bm{U}_{LD})\bigg),\label{add17}\\
\bm{G}_R^{\ast\ast}&=\frac{1}{S_U-S_D}\bigg(S_U\bm{G}_{RD}-S_D\bm{G}_{RU}
+S_DS_U(\bm{U}_{RU}-\bm{U}_{RD})\bigg),\label{add18}\\
\bm{G}_L^{\ast\ast}&=\frac{1}{S_U-S_D}\bigg(S_U\bm{G}_{LD}-S_D\bm{G}_{LU}
+S_DS_U(\bm{U}_{LU}-\bm{U}_{LD})\bigg).\label{add19}
\end{align}
Combining $\eqref{eq43}$ with the relations
in \eqref{eq11}-\eqref{add19} yields
\vspace{-1ex}
\begin{align}\label{eq34}
\bm{F}^\ast&=\frac{1}{S_U-S_D}\bigg(S_U\bm{F}_U^{\ast\ast}-S_D\bm{F}_D^{\ast\ast}-
\frac{2S_LS_R}{S_R-S_L}(\bm{G}_{RU}-\bm{G}_{RD}-\bm{G}_{LU}+\bm{G}_{LD})\bigg),
\\
\label{eq35}
\bm{G}^\ast&=\frac{1}{S_R-S_L}\bigg(S_R\bm{G}_R^{\ast\ast}-S_L\bm{G}_L^{\ast\ast}-
\frac{2S_DS_U}{S_U-S_D}(\bm{F}_{RU}-\bm{F}_{RD}-\bm{F}_{LU}+\bm{F}_{LD})\bigg).
\end{align}

For all other cases with  the certain signs (either positive or negative) of $S_L,S_R,S_D,S_U$, different from the case of $S_L<0<S_R$ and $S_D<0<S_U$,
we can similarly evaluate the above area integrals on the $x=0$ and $y=0$ faces, and  get the  fluxes $\bm{F}^\ast$ and $\bm{G}^\ast$ in the multidimensional Riemann solver
by \eqref{eq11} and \eqref{eq12}.
 The multidimensional fluxes  can be generalized for all situations
by setting \cite{toro}
\begin{equation}\label{eq29}
S_L^{-}=\min(S_L,0),~~S_R^{+}=\max(S_R,0),~~S_D^{-}=\min(S_D,0),~~S_U^{+}=\max(S_U,0).
\end{equation}
If denoting the 2D HLL fluxes as $\bm{F}^{\text{2D-HLL}}$
and $\bm{G}^{\text{2D-HLL}}$ in $x$- and $y$-directions respectively,
then one has their explicit forms
\begin{align}\label{eq41}
\begin{aligned}
&\bm{F}^{\text{2D-HLL}}(\bm{U}_{LD},\bm{U}_{LU},\bm{U}_{RD},\bm{U}_{RU})\\
=&\frac{1}{S_U^{+}-S_D^{-}}\bigg(S_U^{+}\bm{F}_U^{\ast\ast}-S_D^{-}\bm{F}_D^{\ast\ast}-
\frac{2S_L^{-}S_R^{+}}{S_R^{+}-S_L^{-}}(\bm{G}_{RU}-\bm{G}_{RD}-\bm{G}_{LU}+\bm{G}_{LD})\bigg),
\end{aligned}
\\
\label{eq42}
\begin{aligned}
&\bm{G}^{\text{2D-HLL}}(\bm{U}_{LD},\bm{U}_{LU},\bm{U}_{RD},\bm{U}_{RU})\\
=&\frac{1}{S_R^{+}-S_L^{-}}\bigg(S_R^{+}\bm{G}_R^{\ast\ast}-S_L^{-}\bm{G}_L^{\ast\ast}-
\frac{2S_D^{-}S_U^{+}}{S_U^{+}-S_D^{-}}(\bm{F}_{RU}-\bm{F}_{RD}-\bm{F}_{LU}+\bm{F}_{LD})\bigg).
\end{aligned}
\end{align}

Now we begin to study the PCP property of the multidimensional
HLL Riemann solver, which means that the resolved state $\bm{U}^{\ast}$
in the multidimensional Riemann solver is admissible.
Only the case  of $S_L<0<S_R,S_D<0<S_U$ needs to be discussed here,
since the other situations (except for the non-trivial case of $S_L<0<S_R,S_D<0<S_U$)
produce the 1D resolved state which can be easily proved to be PCP according
to \cite{ling}.

Before that, we first give the specific definition of
wave speeds $S_L,S_R,S_D,S_U$.
There exist several different ways to define the wave speeds  in the 1D HLL-type Riemann solvers, see e.g.   \cite{batten,einfeldt,davis}.
If letting $\lambda_A^{(1)}(\bm{U}_{LD})$ and
$\lambda_A^{(4)}(\bm{U}_{LD})$ denote the smallest and largest eigenvalues of Jacobian matrix
$\partial\bm{F}/\partial\bm{U}(\bm{U}_{LD})$, and making corresponding definitions
for the other states in the $x$- and  $y$-directions similarly,
then the wave speeds $S_L, S_R, S_D$ and $S_U$ are given by
\begin{equation}\label{eq2}
\begin{aligned}
&S_L=\alpha\min\big(\lambda_A^{(1)}(\bm{U}_{LD}), \lambda_A^{(1)}(\bm{U}_{RD}), \lambda_A^{(1)}(\bm{U}_{LU}),
\lambda_A^{(1)}(\bm{U}_{RU})\big),\\
&S_R=\alpha\max\big(\lambda_A^{(4)}(\bm{U}_{LD}), \lambda_A^{(4)}(\bm{U}_{RD}), \lambda_A^{(4)}(\bm{U}_{LU}),
\lambda_A^{(4)}(\bm{U}_{RU})\big),\\
&S_D=\alpha\min\big(\lambda_{B}^{(1)}(\bm{U}_{LD}), \lambda_{B}^{(1)}(\bm{U}_{RD}), \lambda_{B}^{(1)}(\bm{U}_{LU}),
\lambda_{B}^{(1)}(\bm{U}_{RU})\big),\\
&S_U=\alpha\max\big(\lambda_{B}^{(4)}(\bm{U}_{LD}), \lambda_{B}^{(4)}(\bm{U}_{RD}), \lambda_{B}^{(4)}(\bm{U}_{RU}),
\lambda_{B}^{(4)}(\bm{U}_{LU})\big),
\end{aligned}
\end{equation}
where $\alpha\ge1$ is to be determined later.
The PCP property of the multidimensional HLL Riemann solver with \eqref{eq2} can be obtained as follows.
\begin{theorem}\label{thm1}
The intermediate state $\bm{U}^{\ast}$ in \eqref{eq43} obtained from the multidimensional
HLL Riemann solver is admissible, i.e.
\begin{equation*}
D^{\ast}>0,~~E^{\ast}>0,~~(E^{\ast})^2-(D^{\ast})^2-|\bm{m}^{\ast}|^2>0,
\end{equation*}
if the wave speeds $S_L,S_R,S_D,S_U$ are taken as \eqref{eq2} with $\alpha=2$.
\end{theorem}

\begin{proof}
Let us assume $S_L<0<S_R,S_D<0<S_U$.
Rewrite $\bm{U}^{\ast}$
as
\begin{equation*}
\bm{U}^{\ast}=\frac{1}{\mathscr{B}}\bigg(S_LS_D\bm{H}_{LD}
-S_RS_D\bm{H}_{RD}-S_LS_U\bm{H}_{LU}+S_RS_U\bm{H}_{RU}\bigg),
\end{equation*}
with $\mathscr{B}=(S_R-S_L)(S_U-S_D)$ and
\begin{equation*}
\begin{aligned}
&\bm{H}_{LD}=\bm{U}_{LD}-\frac{1}{S_L}\bm{F}_{LD}-\frac{1}{S_D}\bm{G}_{LD},~~
\bm{H}_{RD}=\bm{U}_{RD}-\frac{1}{S_R}\bm{F}_{RD}-\frac{1}{S_D}\bm{G}_{RD},\\
&\bm{H}_{LU}=\bm{U}_{LU}-\frac{1}{S_L}\bm{F}_{LU}-\frac{1}{S_U}\bm{G}_{LU},~~
\bm{H}_{RU}=\bm{U}_{RU}-\frac{1}{S_R}\bm{F}_{RU}-\frac{1}{S_U}\bm{G}_{RU}.
\end{aligned}
\end{equation*}
Because $\bm{U}^\ast$ is a convex combination
of $\bm{H}_{LD},\bm{H}_{RD}$, $\bm{H}_{LU},\bm{H}_{RU}$,
it is sufficient to have those
$\bm{H}$-terms to be in the admissible set $\mathcal{G}$.
As an example, consider the term $\bm{H}_{LD}$,
which can be decomposed  into two parts as follows
\begin{equation*}
\bm{H}_{LD}=\bm{U}_{LD}-\frac{1}{S_L}\bm{F}_{LD}-\frac{1}{S_D}\bm{G}_{LD}
=\frac{1}{2}\bigg(\bm{U}_{LD}-\frac{2}{S_L}\bm{F}_{LD}\bigg)
+\frac{1}{2}\bigg(\bm{U}_{LD}-\frac{2}{S_D}\bm{G}_{LD}\bigg).
\end{equation*}
The properties (\romannumeral2) and (\romannumeral3)
in Lemma \ref{lem2} show
the admissibility of
$\bm{U}_{LD}-\frac{2}{S_L}\bm{F}_{LD}$ and $\bm{U}_{LD}-\frac{2}{S_D}\bm{G}_{LD}$ so is $\bm{H}_{LD}$.
Using the similar way can shows that   other $\bm{H}$-terms are also admissible.
The proof is completed.
\end{proof}

\section{Numerical scheme}\label{scheme}
This section presents the first-order accurate PCP finite volume schemes with multidimensional HLL Riemann solver
for the special RHD equations \eqref{eq27}.

Assume that she computational domain is divided into $N\times M$ rectangular cells:
$I_{i,j} =(x_{i-\frac{1}{2}},x_{i+\frac{1}{2}})\times(y_{j-\frac{1}{2}},y_{j+\frac{1}{2}})$ with
the step-sizes $\Delta x_i=x_{i+\frac{1}{2}}-x_{i-\frac{1}{2}},
\Delta y_j=y_{j+\frac{1}{2}}-y_{j-\frac{1}{2}}$
for $i=1,\cdots,N; j=1,\cdots,M$,
and that the time interval $\{t>0\}$ is discretized as:
$t_{n+1}=t_n+\Delta t^n$, $n=0,1,2,\cdots$,
where $\Delta t^n$ is the time step size at $t=t_n$, determined
by the CFL type condition
\begin{equation}\label{eq39}
\Delta t^n\le\sigma\min\limits_{i,j}\left\{\frac{\Delta x_i}
{\max\big(|\lambda_A^{(1)}(\overline{\bm{U}}_{ij}^n)|,
|\lambda_A^{(4)}(\overline{\bm{U}}_{ij}^n)|\big)},\frac{\Delta y_j}
{\max\big(|\lambda_B^{(1)}(\overline{\bm{U}}_{ij}^n)|,|\lambda_B^{(4)}(\overline{\bm{U}}_{ij}^n)|\big)}\right\},
\end{equation}
here the CFL number $\sigma\le1$, and
  $\overline{\bm{U}}_{ij}^{n}$ is the (approximate)
cell average value of $\bm{U}$ at $t_n$  over the cell $I_{i,j}$.

For the RHD system \eqref{eq27}, the finite volume scheme with the Euler forward time discretization
can be formulated as follows
\begin{equation}\label{eq20}
\overline{\bm{U}}_{ij}^{n+1}=\overline{\bm{U}}_{ij}^{n}-\frac{\Delta t^n}{\Delta x_i}
\big(\widehat{\bm{F}}_{i+\frac{1}{2},j}-\widehat{\bm{F}}_{i-\frac{1}{2},j}\big)-\frac{\Delta t^n}{\Delta y_j}
\big(\widehat{\bm{G}}_{i,j+\frac{1}{2}}-\widehat{\bm{G}}_{i,j-\frac{1}{2}}\big),
\end{equation}
where 
$\widehat{\bm{F}}$ and $\widehat{\bm{G}}$ are the numerical fluxes evaluated at
 zone faces corresponding to the $x$- and $y$-directions, respectively.
For the  genuinely multidimensional scheme, by means of Figure \ref{facefluxes}, the fluxes $\widehat{\bm{F}}$ and $\widehat{\bm{G}}$   should be contributed by the 1D Riemann solver at the center of the zone face
and the 2D Riemann solver at the corners of that face.
In our
2D scheme, the flux $\widehat{\bm{F}}_{i+\frac{1}{2},j}$ consists
of three parts: $\bm{F}_{i+\frac{1}{2},j}^{\text{1D-HLL}}$ computed from the 1D HLL Riemann
solver and $\bm{F}_{i+\frac{1}{2},j\pm\frac{1}{2}}^{\text{2D-HLL}}$ computed from the 2D
HLL Riemann solver, where
\begin{equation}\label{eq44}
\begin{aligned}
&\bm{F}^{\text{1D-HLL}}_{i+\frac{1}{2},j}=\bm{F}^{\text{1D-HLL}}(\bm{U}_{i+\frac{1}{2},j}^{L},\bm{U}_{i+\frac{1}{2},j}^{R}),\\
&\bm{F}^{\text{2D-HLL}}_{i+\frac{1}{2},j\pm\frac{1}{2}}
=\bm{F}^{\text{2D-HLL}}\big(\bm{U}^{LD}_{i+\frac{1}{2},j\pm\frac{1}{2}},
\bm{U}^{RD}_{i+\frac{1}{2},j\pm\frac{1}{2}},\bm{U}^{LU}_{i+\frac{1}{2},j\pm\frac{1}{2}},
\bm{U}^{RU}_{i+\frac{1}{2},j\pm\frac{1}{2}}\big),
\end{aligned}
\end{equation}
with $\bm{U}_{i+\frac{1}{2},j}^{L},\bm{U}_{i+\frac{1}{2},j}^{R}$ being the left and right
limited approximations of $\bm{U}$ at the center of the edge $x=x_{i+\frac{1}{2}}$, and
$\bm{U}^{LD}_{i+\frac{1}{2},j\pm\frac{1}{2}},
\bm{U}^{RD}_{i+\frac{1}{2},j\pm\frac{1}{2}},\bm{U}^{LU}_{i+\frac{1}{2},j\pm\frac{1}{2}},
\bm{U}^{RU}_{i+\frac{1}{2},j\pm\frac{1}{2}}$ being the left-down, right-down, left-up and right-up
limited approximations of $\bm{U}$ at the node $(x_{i+\frac{1}{2}},y_{j\pm\frac{1}{2}})$.
Those limited approximations can be obtained
by using the initial reconstruction technique and $\{\overline{\bm{U}}_{ij}^{n}\}$ respectively.
For example, for the first-order accurate scheme, they can be calculated from the reconstructed
piecewise constant function
\begin{equation}\label{add21}
\begin{aligned}
&\bm{U}_{i+\frac{1}{2},j}^{L}=\overline{\bm{U}}_{ij},
~~~~~~~~\bm{U}_{i+\frac{1}{2},j}^{R}=\overline{\bm{U}}_{i+1,j},
~~~~~\bm{U}^{LD}_{i+\frac{1}{2},j+\frac{1}{2}}=\overline{\bm{U}}_{ij},\\
&\bm{U}^{RD}_{i+\frac{1}{2},j+\frac{1}{2}}=\overline{\bm{U}}_{i+1,j},
~\bm{U}^{LU}_{i+\frac{1}{2},j+\frac{1}{2}}=\overline{\bm{U}}_{i,j+1},
~~\bm{U}^{RU}_{i+\frac{1}{2},j+\frac{1}{2}}=\overline{\bm{U}}_{i+1,j+1}.
\end{aligned}
\end{equation}
%
In summary, our final numerical fluxes $\widehat{\bm{F}}$ and $\widehat{\bm{G}}$ are given as follows
\begin{align}
\widehat{\bm{F}}_{i+\frac{1}{2},j}&=\frac{\Delta t^n}{2\Delta y_j}
\big(S_{U,i+\frac{1}{2},j-\frac{1}{2}}^{+}\bm{F}^{\text{2D-HLL}}_{i+\frac{1}{2},j-\frac{1}{2}}
-S_{D,i+\frac{1}{2},j+\frac{1}{2}}^{-}\bm{F}^{\text{2D-HLL}}_{i+\frac{1}{2},j+\frac{1}{2}}\big)\notag\\
&~~~+\left(1-\frac{\Delta t^n}{2\Delta y_j}
(S_{U,i+\frac{1}{2},j-\frac{1}{2}}^{+}-S_{D,i+\frac{1}{2},j+\frac{1}{2}}^{-})\right)
\bm{F}_{i+\frac{1}{2},j}^{\text{1D-HLL}},\label{add23}\\
\widehat{\bm{G}}_{i,j+\frac{1}{2}}&=\frac{\Delta t^n}{2\Delta x_i}
\big(S_{R,i-\frac{1}{2},j+\frac{1}{2}}^{+}\bm{G}^{\text{2D-HLL}}_{i-\frac{1}{2},j+\frac{1}{2}}
-S_{L,i+\frac{1}{2},j+\frac{1}{2}}^{-}\bm{G}^{\text{2D-HLL}}_{i+\frac{1}{2},j+\frac{1}{2}}\big)\notag\\
&~~~+\left(1-\frac{\Delta t^n}{2\Delta x_i}
(S_{R,i-\frac{1}{2},j+\frac{1}{2}}^{+}-S_{L,i+\frac{1}{2},j+\frac{1}{2}}^{-})\right)
\bm{G}_{i,j+\frac{1}{2}}^{\text{1D-HLL}},\label{add24}
\end{align}
where $S_L^{-},S_R^{+},S_D^{-},S_U^{+}$ are defined in \eqref{eq29}.

\begin{figure}[H]
\begin{minipage}{0.55\textwidth}
\centering
\includegraphics[width=3.6in]{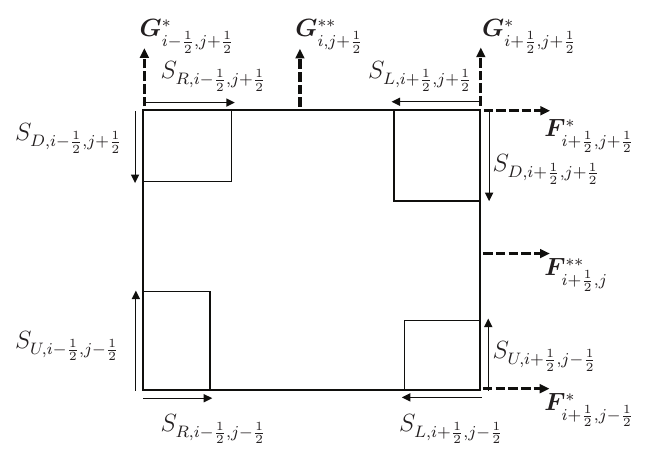}
\end{minipage}
\begin{minipage}{0.45\textwidth}
\centering
\includegraphics[width=2.6in]{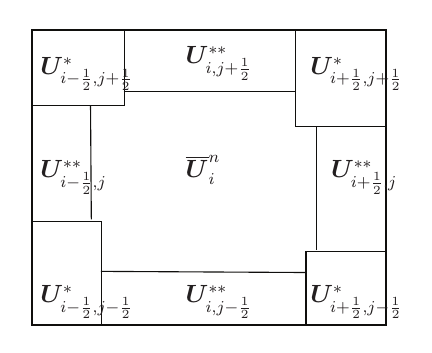}
\end{minipage}
\caption{Illustration for the numerical fluxes (left) and approximate solution (right) for the numerical scheme
\eqref{eq20}.}
\label{inter}
\end{figure}

Now, let us discuss the PCP property of the scheme \eqref{eq20} with the numerical
fluxes  \eqref{add23} and \eqref{add24}.
Following the design of the 1D and 2D HLL Riemann solvers,
under the CFL condition \eqref{eq39} with $\sigma\le\frac{1}{2}$, the scheme
\eqref{eq20} can be written as an exact integration of those two approximate Riemann solvers
over the cell $I_{i,j}$.
For the the non-trivial case: $S_L<0<S_R,S_D<0<S_U$, the updated solution
$\overline{\bm{U}}_{ij}^{n+1}$ can be reformulated as a convex combination of nine terms:
$\overline{\bm{U}}_{ij}^n,\bm{U}_{i-\frac{1}{2},j-\frac{1}{2}}^{\ast},\bm{U}_{i+\frac{1}{2},j-\frac{1}{2}}^{\ast},
\bm{U}_{i-\frac{1}{2},j+\frac{1}{2}}^{\ast},\bm{U}_{i-\frac{1}{2},j+\frac{1}{2}}^{\ast}$ and
$\bm{U}_{i-\frac{1}{2},j}^{\ast\ast},\bm{U}_{i+\frac{1}{2},j}^{\ast\ast},
\bm{U}_{i,j-\frac{1}{2}}^{\ast\ast},\bm{U}_{i,j+\frac{1}{2}}^{\ast\ast}$, which is illustrated in
Figure \ref{inter}. The terms with superscript
``$\ast$'' are the approximate Riemann solutions obtained from the multidimensional HLL
Riemann solver and
the terms with superscript ``$\ast\ast$'' are the approximate Riemann solutions obtained from the
1D HLL Riemann solver.
Each term in the convex combination is admissible, see Section \ref{multi-hlle},
and thus the numerical solution
$\overline{\bm{U}}_{ij}^{n+1}$ to the first-order scheme \eqref{eq20} with \eqref{add21} is also admissible.
We conclude such result in the following theorem.
\begin{theorem}\label{thm2}
If $\{\overline{\bm{U}}_{ij}^{n}\in \mathcal G,\forall i=1,
\cdots,N; j=1,\cdots,M\}$ and the wave speeds $S_D,S_U,S_L,S_R$ are estimated by \eqref{eq2},
then $\overline{\bm{U}}_{ij}^{n+1}$ obtained by
the first-order scheme \eqref{eq20} with \eqref{add21} and the multidimensional Riemann solver
belongs to the admissible state set $\mathcal G$
under the time step restriction \eqref{eq39}
with   $\sigma\le\frac{1}{2}$.
\end{theorem}

\section{Numerical tests}\label{num}
This section conducts several numerical experiments on the 2D ultra-relativistic RHD problems with large Lorentz factor, or strong discontinuities, or low rest-mass density or pressure, to verify the
accuracy, robustness, and effectiveness of the present PCP scheme.
It is worth remarking that those ultra-relativistic RHD problems seriously challenge the numerical scheme. Unless otherwise stated, all the computations are restricted to the equation of state
\eqref{eq15} with the adiabatic index $\Gamma=5/3$, and
the time step size $\Delta t^n$ determined by \eqref{eq39} with the CFL number $\sigma=0.45$.

\begin{example}[Sine wave propagation]\label{exam1} \rm
This problem is used to test the accuracy of our PCP finite volume scheme. Its 
exact solution is
$$(\rho,u,v,p)(t,x,y)=(1+\delta\sin(2\pi(x+y-0.99\sqrt{2}t)),
0.99/\sqrt{2},0.99/\sqrt{2},0.01),\ t\geq 0$$
which describe an RHD sine wave propagating periodically in the domain $\Omega=(0,1)\times (0,1)$ at
an angle $45^\circ$ with the $x$-axis.
The computational domain is divided into $N\times N$ uniform cells, $\delta=0.99999$,  and the periodic boundary
conditions are specified on the boundary of $\Omega$.  Table \ref{table1} lists the $\ell^1,\ell^2$ and
$\ell^\infty$ errors at $t=0.1$ and orders of convergence obtained from our first-order multidimensional
PCP scheme. The results show the expected PCP performance.
\end{example}
\begin{table}[H]
\centering
\begin{tabular}{c||c|c||c|c||c|c}
\hline
$N$ & $\ell^1$ error & $\ell^1$ order & $\ell^2$ error & $\ell^2$ order & $\ell^{\infty}$ error & $\ell^{\infty}$ order   \\
\hline
 20 &  5.521E-02 &  ----- &  6.146E-02 &  ----- &  8.683E-02 &  -----\\
 40 &  2.705E-02 &  1.029 &  3.003E-02 &  1.033 &  4.241E-02 &  1.034\\
 80 &  1.338E-02 &  1.016 &  1.486E-02 &  1.015 &  2.101E-02 &  1.013\\
160 &  6.710E-03 &  0.996 &  7.452E-03 &  0.996 &  1.054E-02 &  0.996\\
320 &  3.359E-03 &  0.998 &  3.731E-03 &  0.998 &  5.277E-03 &  0.998\\
\hline
\end{tabular}
\caption{\small Example \ref{exam1}: Errors and orders of convergence for the rest-mass density $\rho$ at $t=0.1$
obtained by using the first-order PCP scheme with the mesh of $N\times N$ uniform cells.}
\label{table1}
\end{table}

\begin{example}[Relativistic isentropic vortex]\label{exam2} \rm
It is a 2D relativistic isentropic vortex problem constructed first in \cite{ling}, where
the vortex in the space-time coordinate system $(t,x,y)$ moves with a constant speed of
magnitude $w$ in $(-1,-1)$ direction.
The time-dependent solution $(\rho,u,v,p)$ at time $t\geq 0$
is given as follows
\begin{align*}
&\rho=(1-\alpha e^{1-r^2})^{\frac{1}{\Gamma-1}},\quad p=\rho^\Gamma,\\
&u=\frac{1}{1-\frac{w(u_0+v_0)}{\sqrt{2}}}\left[\frac{u_0}{\gamma}-\frac{w}{\sqrt{2}}+\frac{\gamma w^2}{2(\gamma+1)}(u_0+v_0)\right],\\
&v=\frac{1}{1-\frac{w(u_0+u_0)}{\sqrt{2}}}\left[\frac{v_0}{\gamma}-\frac{w}{\sqrt{2}}+\frac{\gamma w^2}{2(\gamma+1)}(u_0+v_0)\right],
\end{align*}
where
\begin{align*}
&\gamma=\frac{1}{\sqrt{1-w^2}},\quad r=\sqrt{x_0^2+y_0^2},\quad (u_0,v_0)=(-y_0,x_0)f,\\
&\alpha=\frac{(\Gamma-1)}{8\Gamma}\pi^2\epsilon^2,\quad\beta=\dfrac{\Gamma^2\alpha e^{1-r^2}}{2\Gamma-1-\Gamma\alpha e^{1-r^2}},
\quad f=\sqrt{\frac{\beta}{1+\beta r^2}},\\
&x_0=x+\frac{\gamma-1}{2}(x+y)+\frac{\gamma tw}{\sqrt{2}},~~y_0=y+\frac{\gamma-1}{2}(x+y)+\frac{\gamma tw}{\sqrt{2}}.
\end{align*}
Our computations are performed in the domain
$\Omega=(-5,5)\times (-5,5)$ with the adiabatic index $\Gamma=1.4$, $w=0.5\sqrt{2}$,
 the vortex strength $\epsilon=10.0828$, and the periodic boundary conditions. In this case, the
lowest density and lowest pressure are $7.8\times 10^{-15}$ and
$1.78\times 10^{-20}$, respectively.

Table \ref{table2} gives the errors of the rest-mass density at $t=1$ and the orders of convergence. It is clear to see that our multidimensional PCP scheme achieves the expected
accuracy and preserves the positivity of the density and pressure simultaneously.
\end{example}
\begin{table}[H]
\centering
\begin{tabular}{c||c|c||c|c||c|c}
\hline
$N$ & $\ell^1$ error & $\ell^1$ order & $\ell^2$ error & $\ell^2$ order & $\ell^{\infty}$ error &
$\ell^{\infty}$ order   \\
\hline
 20 &  1.566E-02 &  ----- &  5.026E-02 &  ----- &  3.285E-01 &  -----\\
 40 &  8.881E-03 &  0.818 &  2.855E-02 &  0.816 &  2.061E-01 &  0.672\\
 80 &  4.779E-03 &  0.894 &  1.550E-02 &  0.881 &  1.189E-01 &  0.794\\
160 &  2.495E-03 &  0.938 &  8.212E-03 &  0.916 &  6.644E-02 &  0.839\\
320 &  1.280E-03 &  0.963 &  4.260E-03 &  0.947 &  3.528E-02 &  0.913\\
\hline
\end{tabular}
\caption{\small Example \ref{exam2}: Errors and orders of convergence for rest-mass
density at $t=1$ obtained by using the first-order scheme with the mesh of $N\times N$ uniform cells.}
 \label{table2}
\end{table}

\begin{example}[Explosion]\label{exam3} \rm
It is used to test the multi-dimensionality of our scheme.
Initially, the rest fluid with a unit rest-mass density is  in the domain  $\Omega=(-0.5,0.5)\times (-0.5,0.5)$.
The pressure is set as 20 inside a circle of radius 1/10, while
a smaller pressure of 0.1 is given all over outside the circle.
 Figure \ref{fig8}  plots the contours and cut lines along $y$-axis and $y=x$
of the rest-mass density at $t=0.1$ obtained by using
our scheme with the multidimensional Riemann solver
on the mesh of $64\times64$ uniform cells.
For a comparison, Figures \ref{fig8b} gives the numerical solutions
obtained by using
the scheme with the 1D Riemann solver.
It can be clearly seen from them that the results obtained by
our scheme with the 2D Riemann solver preserve the spherical symmetry better.
\end{example}


\begin{figure}[H]
\centering
\includegraphics[width=3.0in]{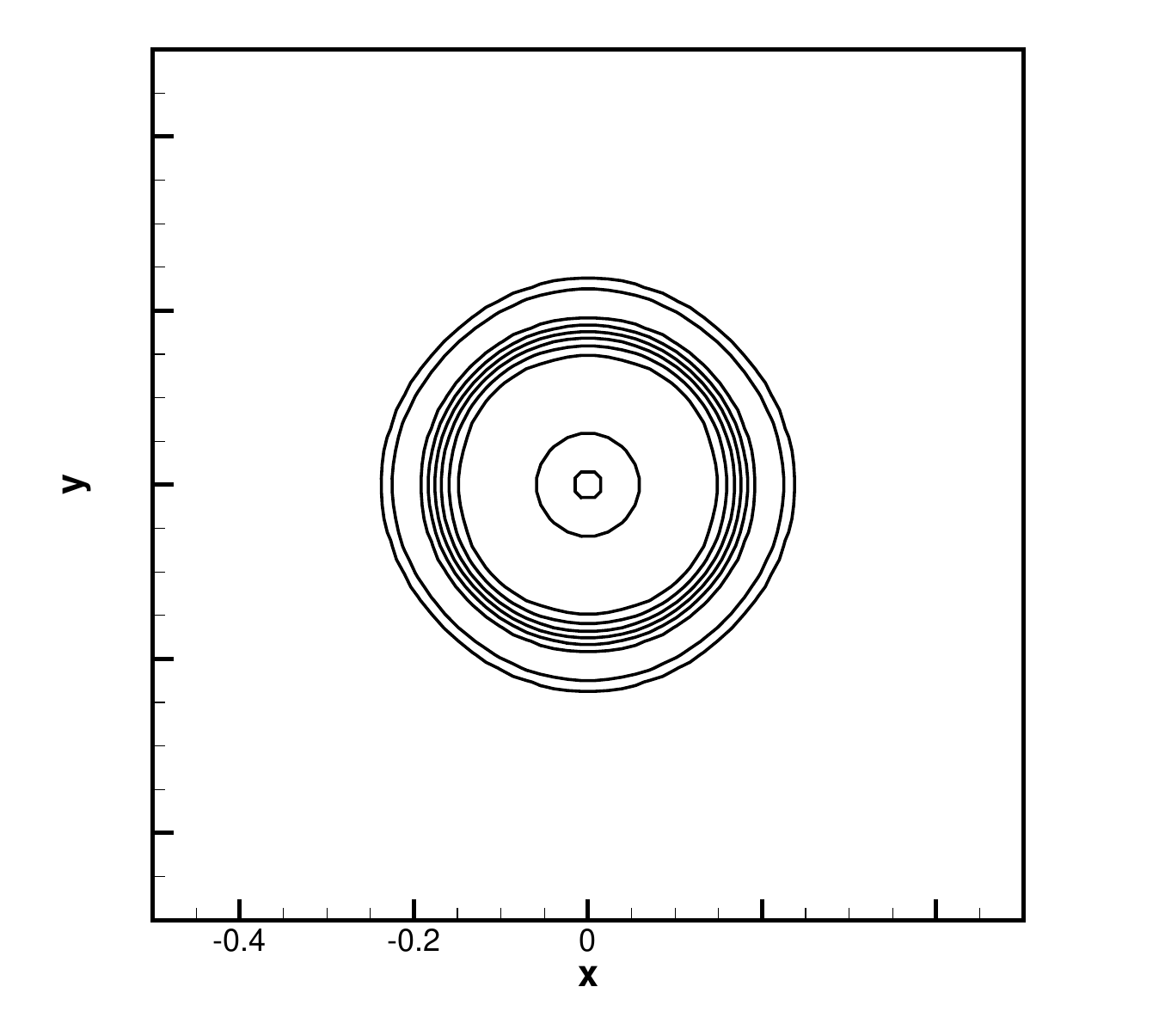}
\includegraphics[width=3.0in]{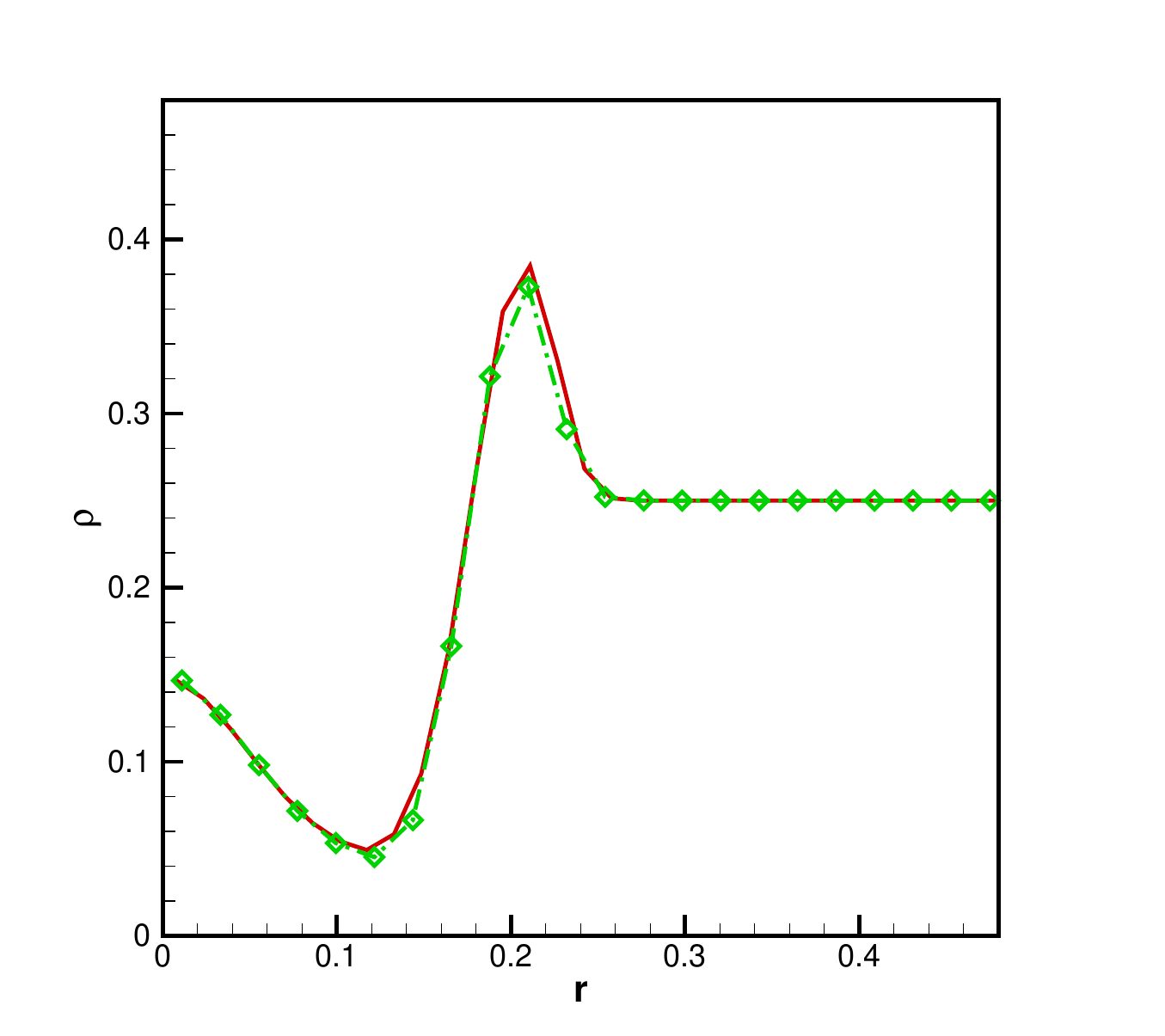}
\caption{\small Example \ref{exam3}: The rest-mass density at $t = 0.1$
obtained our scheme.
Left: The contours with eight equally spaced contour lines;
right: the cut lines of $\rho$ at $y$-axis (solid line) and $y=x$ (dashed line with the symbol ``$\diamond$'').}
\label{fig8}
\end{figure}

\begin{figure}[H]
\centering
\includegraphics[width=3.0in]{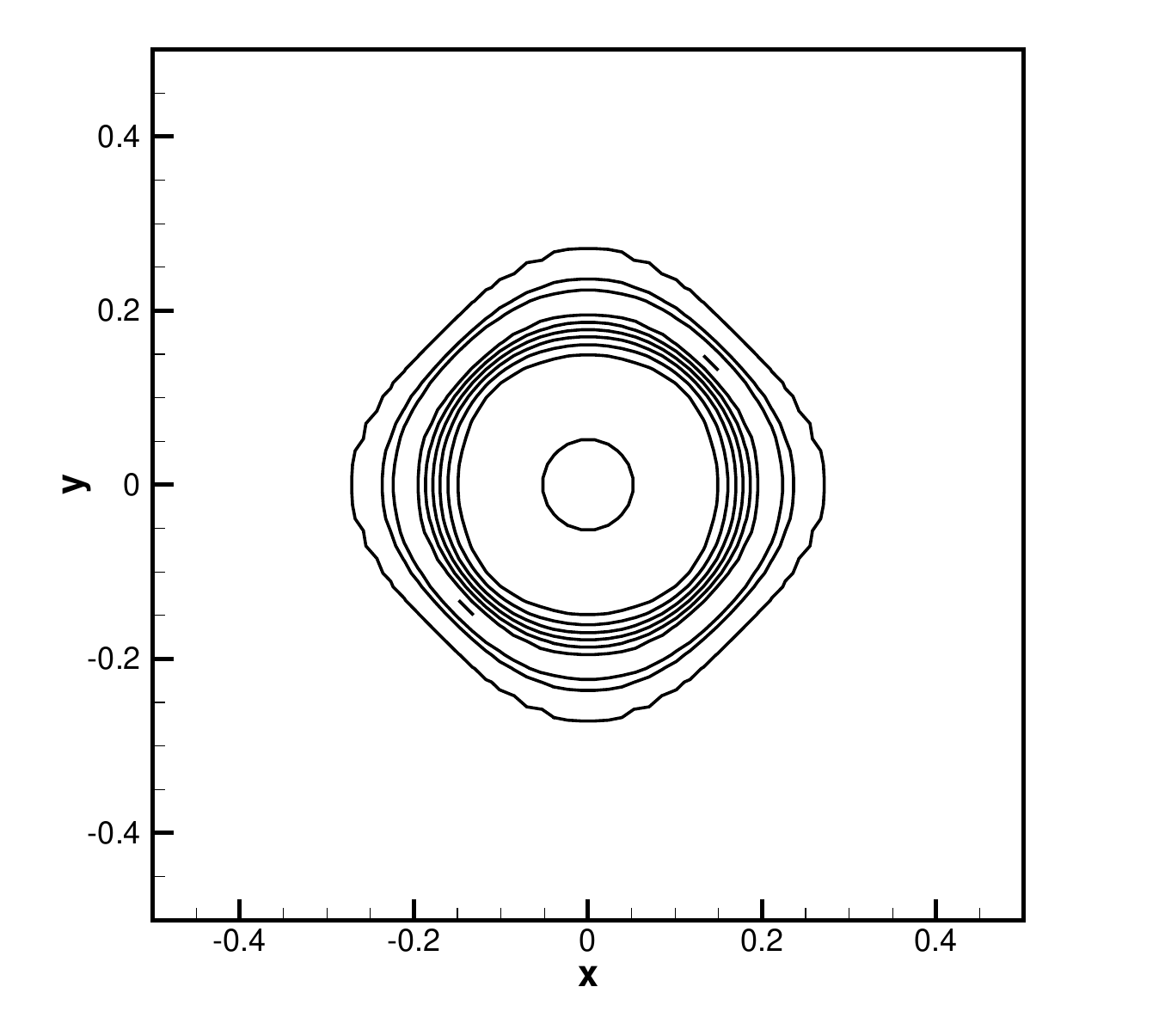}
\includegraphics[width=3.0in]{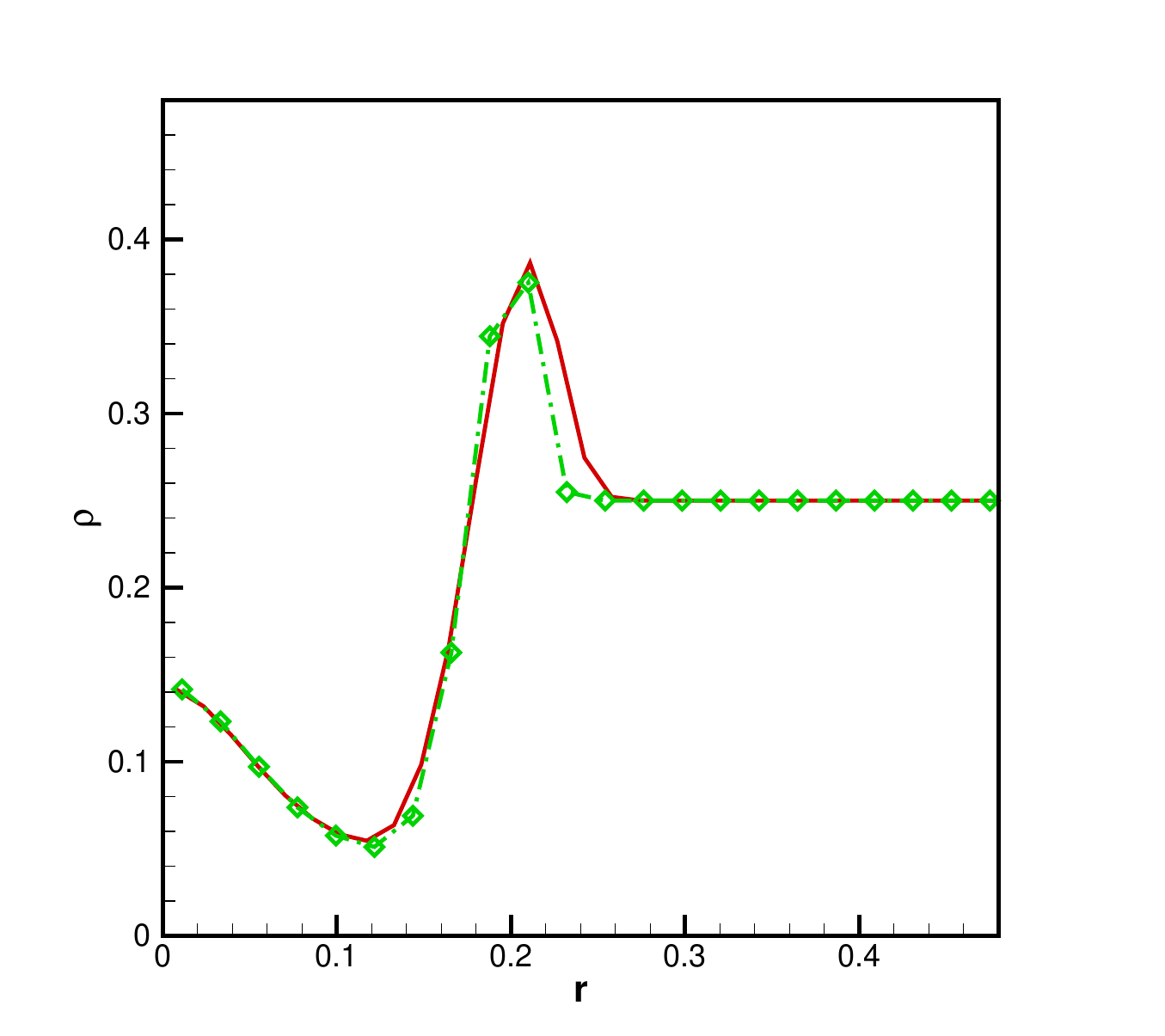}
\caption{\small Same as Figure \ref{fig8}, except for the 1D
HLL Riemann solver.}
\label{fig8b}
\end{figure}

\begin{example}[Riemann problem I]\label{exam4} \rm
The example solves the 2D Riemann problem \cite{zanna,wu2017}.
The initial data are given  by
\begin{equation*}
(\rho,u,v,p)(0,x,y)=\begin{cases}
(0.1,0,0,0.01),& x>0,~y>0,\\
(0.1,0.99,0,1),& x<0,~y>0,\\
(0.5,0,0,1),& x<0,~y<0,\\
(0.1,0,0.99,1),& x>0,~y<0,
\end{cases}
\end{equation*}
where both the left and bottom discontinuities are the contact discontinuities with a jump in the transverse
velocity, while both the right and top discontinuities are not simple waves.
It is necessary to remark that this case is different from that in \cite{yang2012direct}.

The computational domain $\Omega$ is taken as $(-1,1)\times (-1,1)$ and is divided into a uniform mesh with $400\times400$ cells.
Figure \ref{fig2} displays the contours of the rest-mass density logarithm $\ln\rho$ and the pressure
logarithm $\ln p$ at $t=0.8$ obtained by using the first-order multidimensional PCP scheme. We can see that
the four initial discontinuities interact each other and form
two reflected curved shock waves, an elongated jet-like spike.
It is worth mentioning that a non-PCP scheme fails when simulating this problem.
\end{example}
\vspace{-4ex}
\begin{figure}[H]
\centering
\includegraphics[width=3.0in]{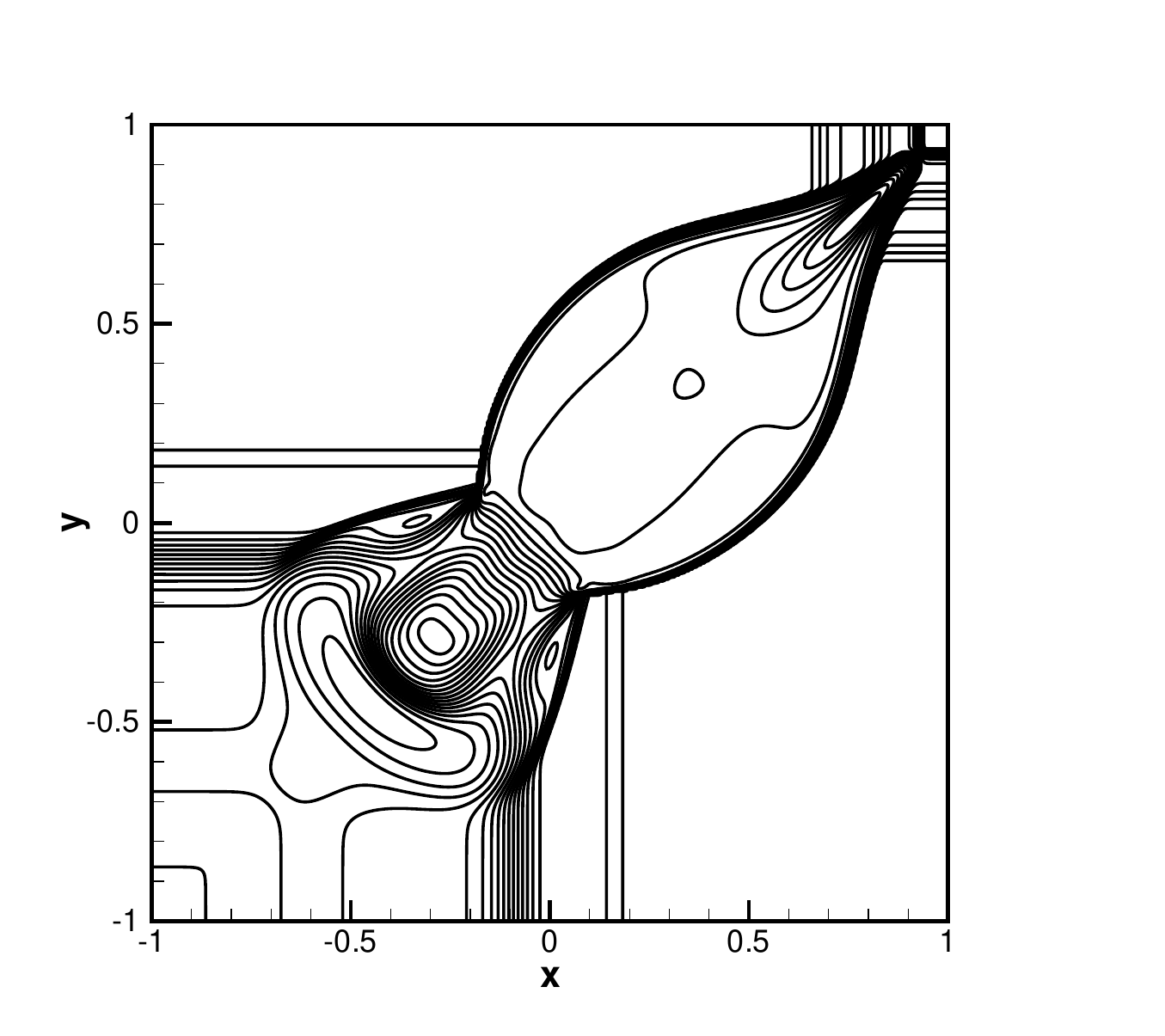}
\includegraphics[width=3.0in]{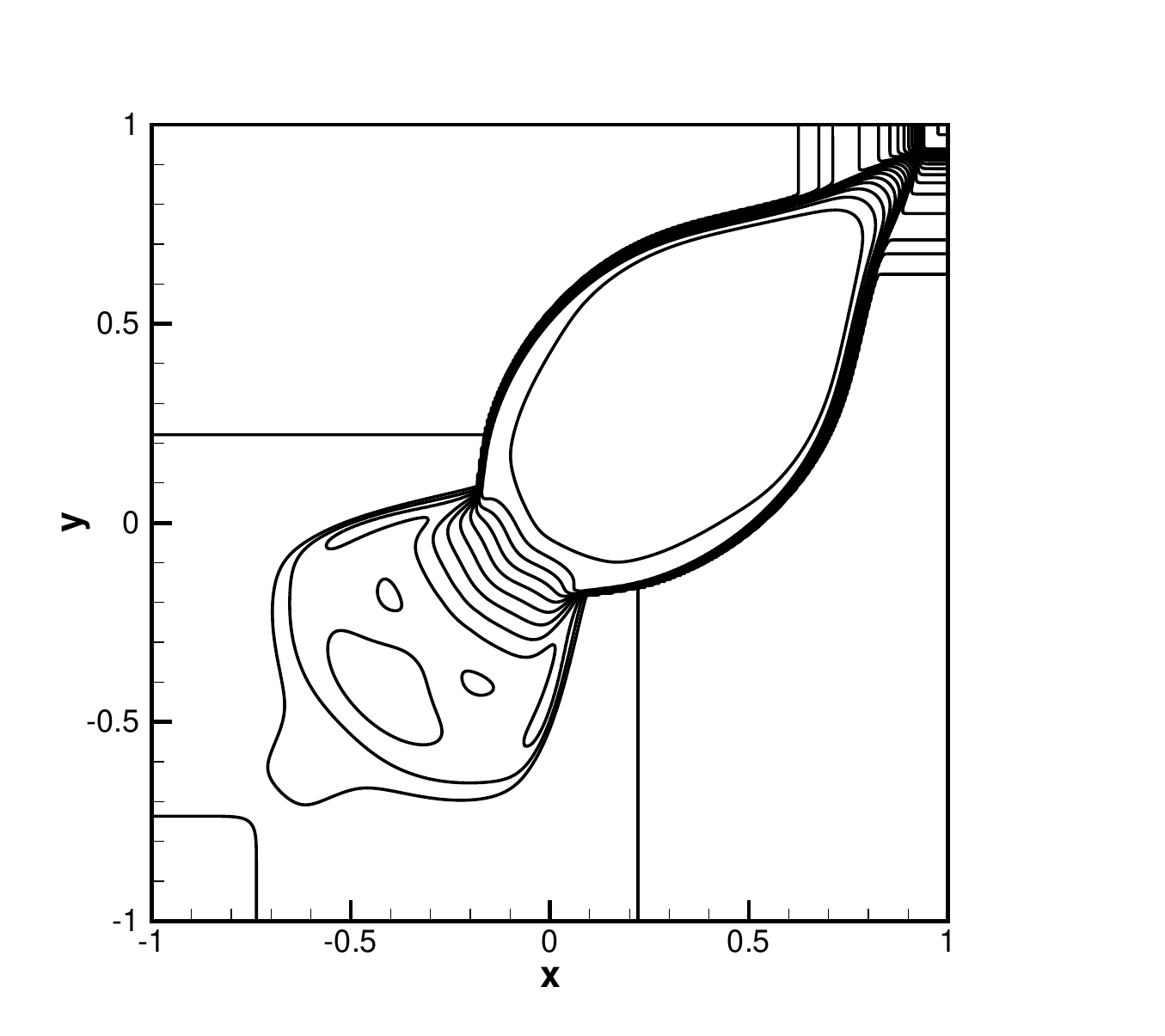}
\caption{\small Example \ref{exam4}: The contours of the density logarithm $\ln\rho$ (left)
and the pressure logarithm $\ln p$ (right) at $t=0.8$. 30 equally spaced contour lines are used.}
\label{fig2}
\end{figure}

\begin{example}[Riemann problem II]\label{exam5} \rm
The initial data of the second Riemann problem   are
\begin{equation*}
(\rho,u,v,p)(0,x,y)=\begin{cases}
(0.1,0,0,20),& x>0,~y>0,\\
(\widetilde{\rho},\widetilde{u},0,0.05),& x<0,~y>0,\\
(0.01,0,0,0.05),& x<0,~y<0,\\
(\widetilde{\rho},0,\widetilde{u},0.05),& x>0,~y<0,
\end{cases}
\end{equation*}
with $\widetilde{\rho}=0.00414329639576$ and $\widetilde{u}=0.9946418833556542$,
In this problem, the left and lower initial discontinuities are the contact discontinuities, while
the upper and right are shock waves with a speed of $-0.66525606186639$.
As the time increases, the maximal value of the fluid velocity becomes very large and close to the
speed of light, which leads to the numerical simulation more challenging.
Figure \ref{fig3} displays the contours of the rest-mass density logarithm $\ln\rho$ and
the pressure logarithm $\ln p$ at $t=0.8$ obtained by using the first-order PCP scheme with
the multidimensional
HLL Riemann solver on the mesh of $400\times400$ uiniform cells.
for $\Omega=(-1,1)\times(-1,1)$.
The interaction of four initial discontinuities results in the distortion of the initial
shock waves and the formation of a ``mushroom cloud'' starting from the point (0,0)
and expanding to the left bottom region.
\end{example}
\vspace{-4ex}
\begin{figure}[H]
\begin{minipage}{0.5\textwidth}
\centering
\includegraphics[width=3.0in]{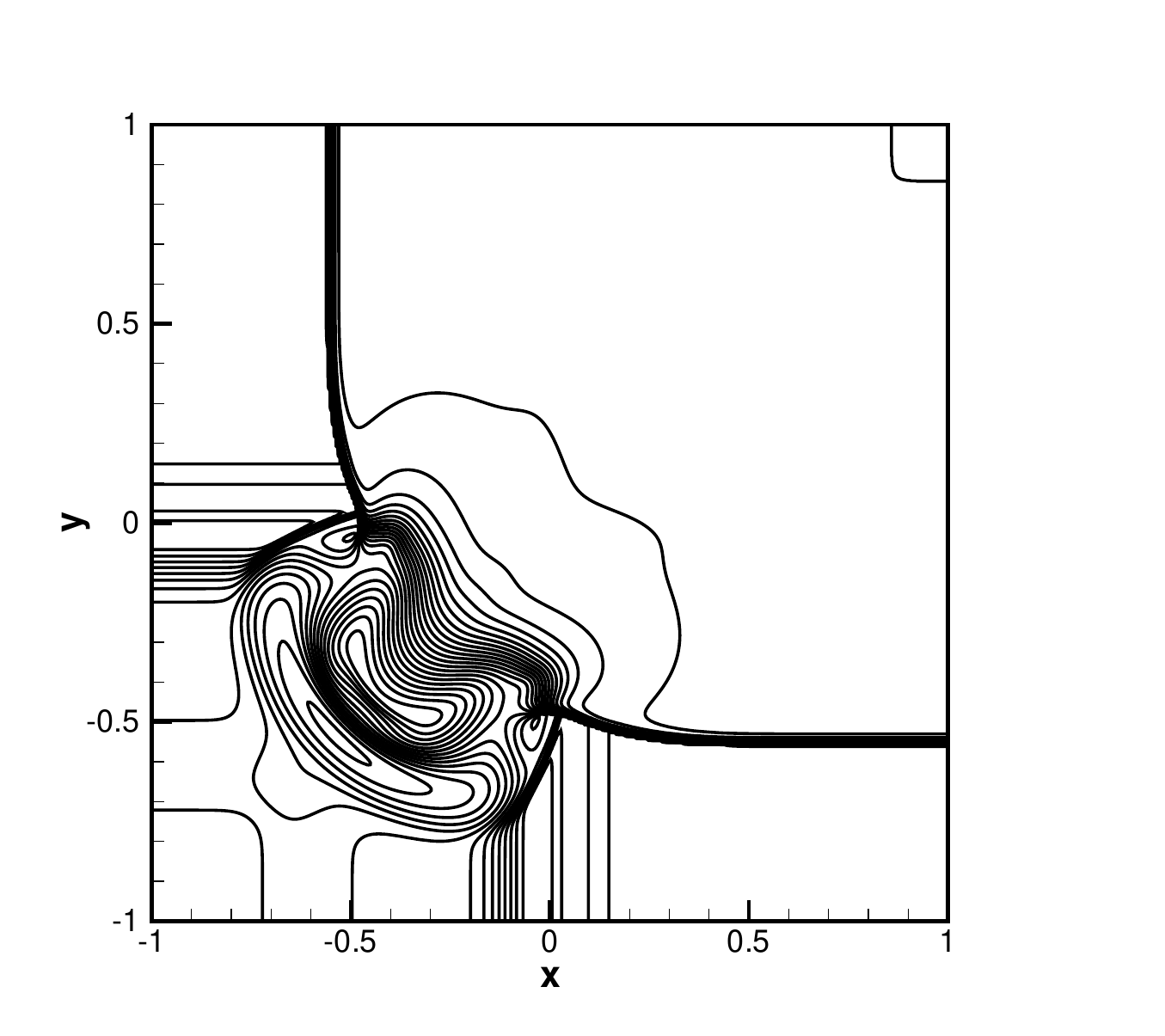}
\end{minipage}
\begin{minipage}{0.5\textwidth}
\centering
\includegraphics[width=3.0in]{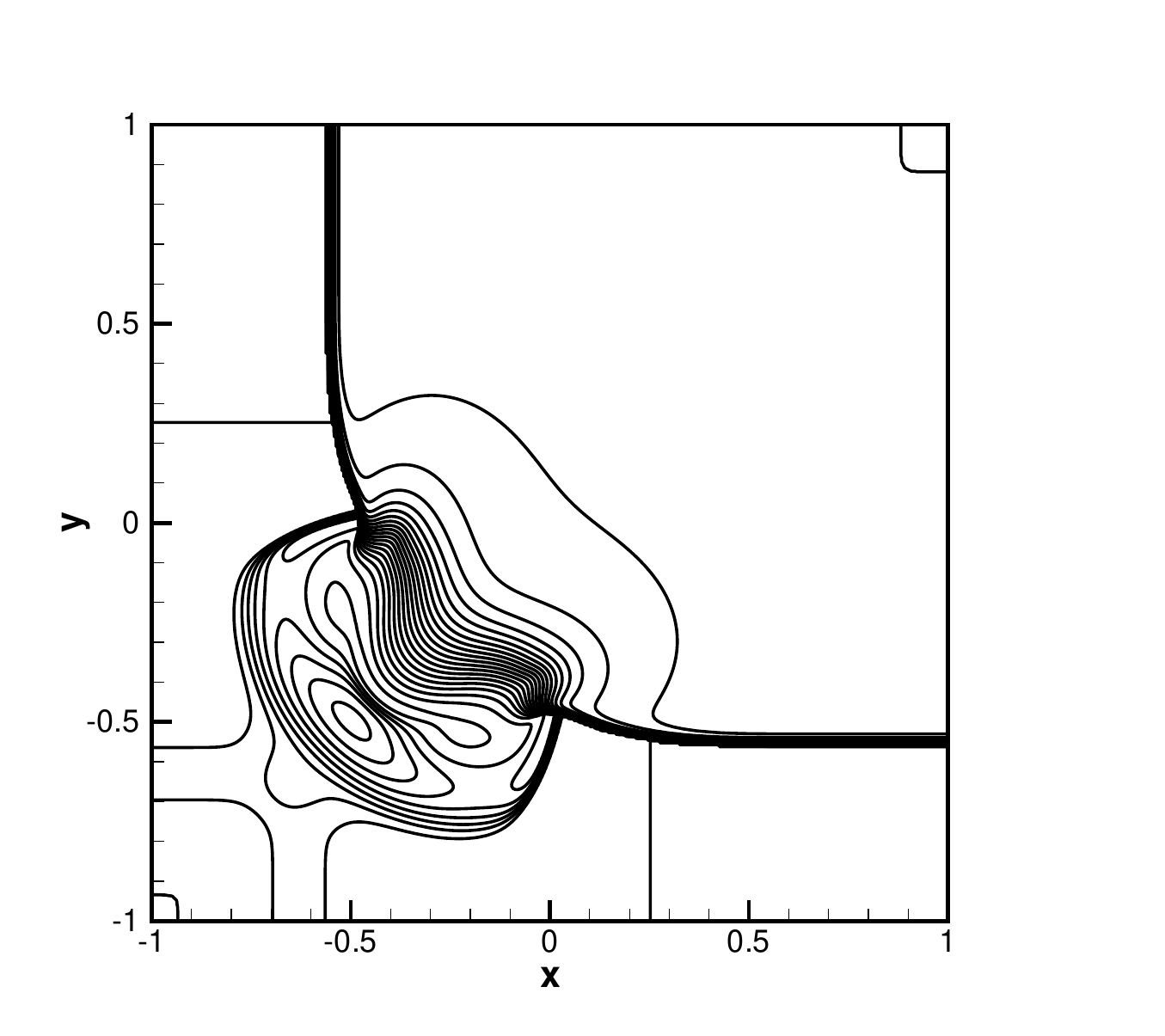}
\end{minipage}
\caption{\small Example \ref{exam5}: The contours of the density logarithm $\ln\rho$ (left)
and the pressure
logarithm $\ln p$ (right) at $t=0.8$. 25 equally spaced contour lines are shown.}
\label{fig3}
\end{figure}

\begin{example}[Relativistic jets]\label{exam6} \rm
The last example is to simulate two high-speed relativistic jet flows \cite{wu2017}, which are ubiquitous in the extragalactic radio sources associated with the active galactic nuclei, and the most compelling case for a special relativistic phenomenon.
It is very challenging to simulate such jet flows because there may appear the strong relativistic shock waves, shear waves, interface instabilities,
and ultra-relativistic regions,  as well as high speed
jets etc.

First, we consider a pressure-matched hot jet model, in which
the relativistic effects from the large beam internal energies are
important and comparable to the effects from the fluid velocity
near the speed of light because the classical beam Mach number $M_b$ is near the minimum Mach number for given beam speed
$v_b$.
Initially, the computational domain
$(0,12)\times(0, 30)$ is filled with a static uniform medium with an
unit rest-mass density, and a light relativistic jet is injected in the
$y$-direction through the inlet part $|x|\le0.5$ on the bottom
boundary $(y=0)$ with a high speed $v_b$, a rest-mass density of 0.01, and
a pressure equal to the ambient pressure.
The  fixed inflow beam condition
is specified on the nozzle $\{y=0,|x|\le0.5\}$, the reflecting boundary
condition is specified at $x=0$, whereas the
outflow boundary conditions are on other boundaries.
The following three different configurations are considered:
\begin{enumerate}[(\romannumeral1)]
\item $v_b=0.99$ and $M_b=1.72$, corresponding to the case of $\gamma\approx7.089$ and
$M_r\approx9.971$.
\item $v_b=0.999$ and $M_b=1.72$, corresponding to the case of $\gamma\approx22.366$ and
$M_r\approx31.316$.
\item $v_b=0.9999$ and $M_b=1.72$, corresponding to the case of $\gamma\approx70.712$ and
$M_r\approx98.962$.
\end{enumerate}
Here  $M_r:=M_b\gamma/\gamma_s$ denotes  the relativistic Mach number
with $\gamma_s=1/\sqrt{1-c_s^2}$ being the Lorentz factor associated with the
local sound speed.

As $v_b$ becomes much closer to the speed of light, the simulation of the jet becomes more challenging.
Figures \ref{fig4} and \ref{fig5} display the schlieren images of the rest-mass density logarithm
$\ln\rho$ and the
pressure logarithm $\ln p$ within the domain $[-12,12]\times[0,30]$ at $t=30$ obtained by using
the first order scheme with multidimensional HLL
Riemann solver with $240\times600$  uniform cells
for the computational domain $(0,12)\times(0,30)$.
\begin{figure}[H]
\centering
\includegraphics[width=1.5in]{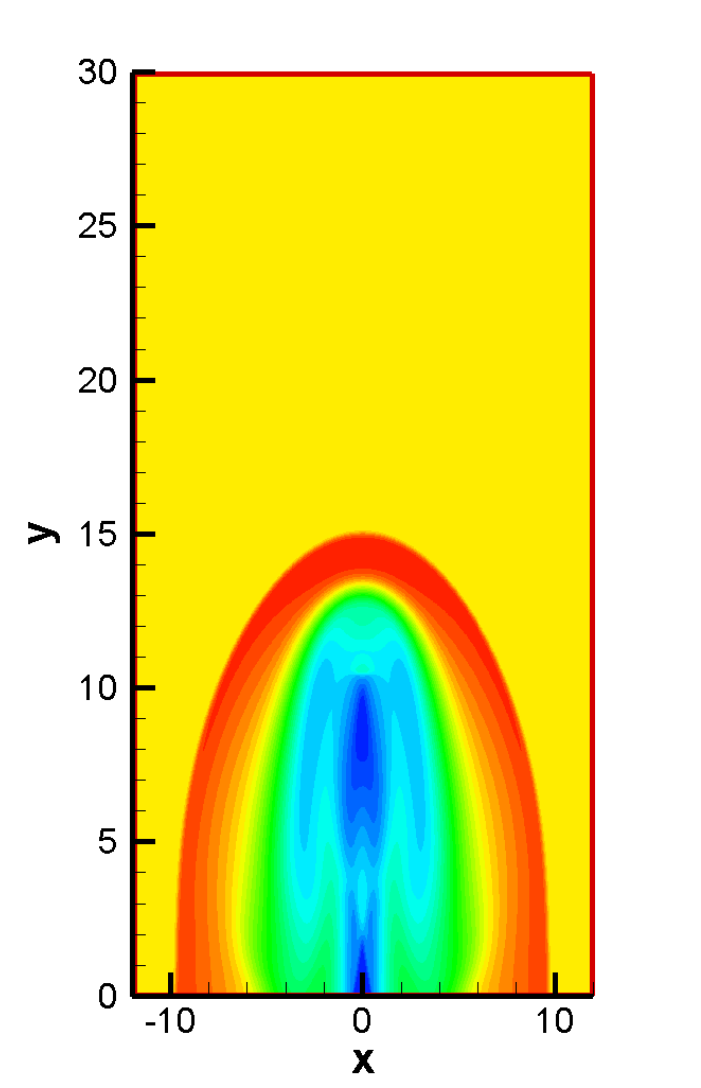}
\includegraphics[width=1.5in]{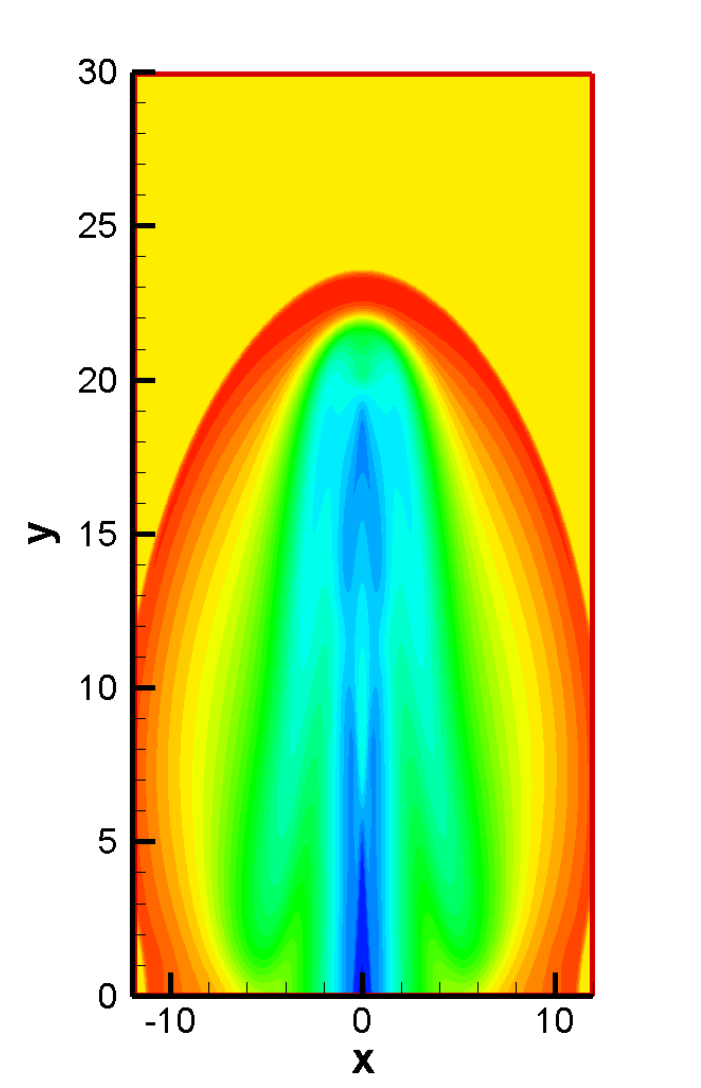}
\includegraphics[width=1.5in]{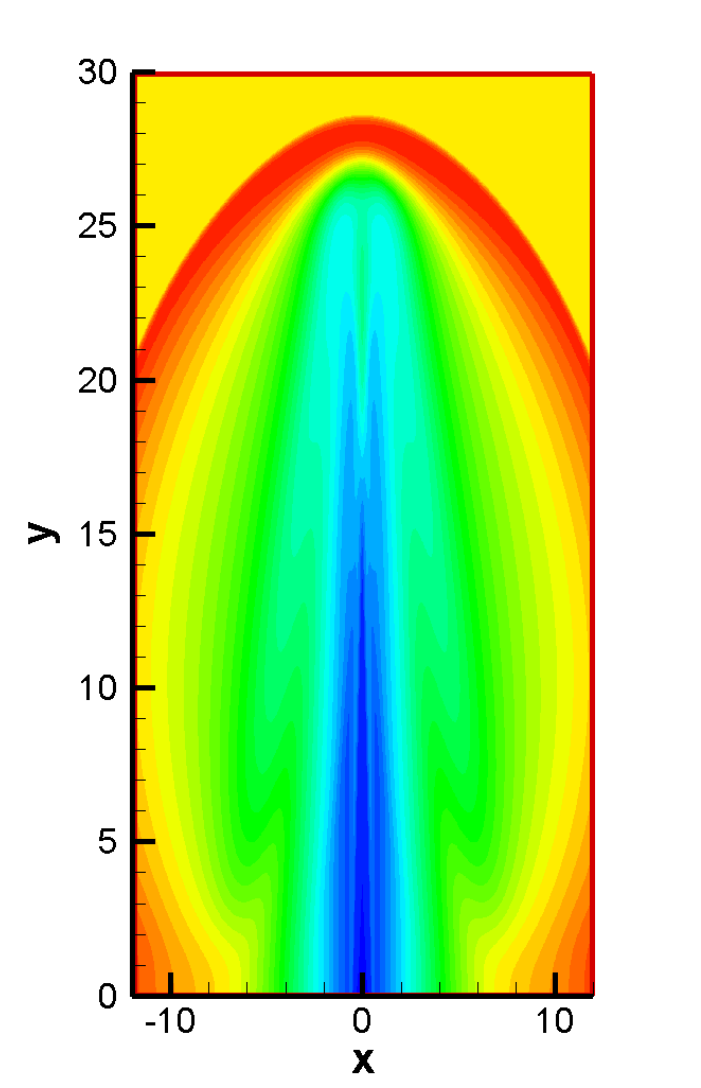}
\caption{\small Example \ref{exam6}: Schlieren images of rest-mass density logarithm $\ln\rho$
at $t=30$ for the hot jet model obtained by the first order scheme with $240\times600$ uniform cells. From left to right:
configurations (\romannumeral1), (\romannumeral2) and (\romannumeral3).}
\label{fig4}
\end{figure}

\begin{figure}[H]
\centering
\includegraphics[width=1.5in]{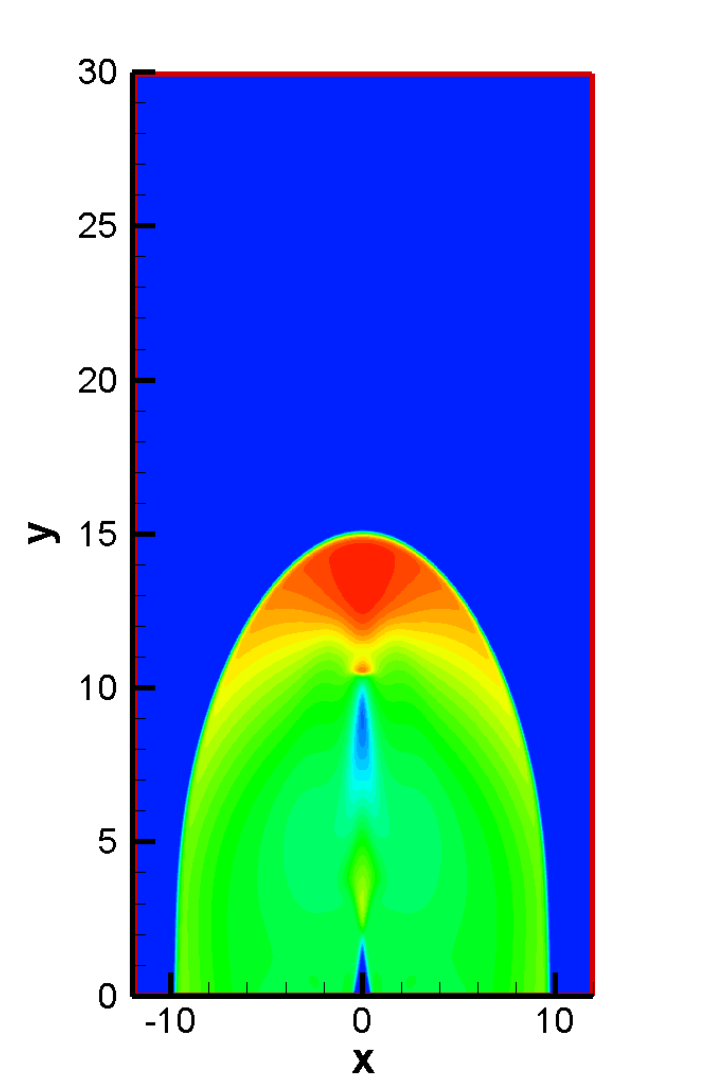}
\includegraphics[width=1.5in]{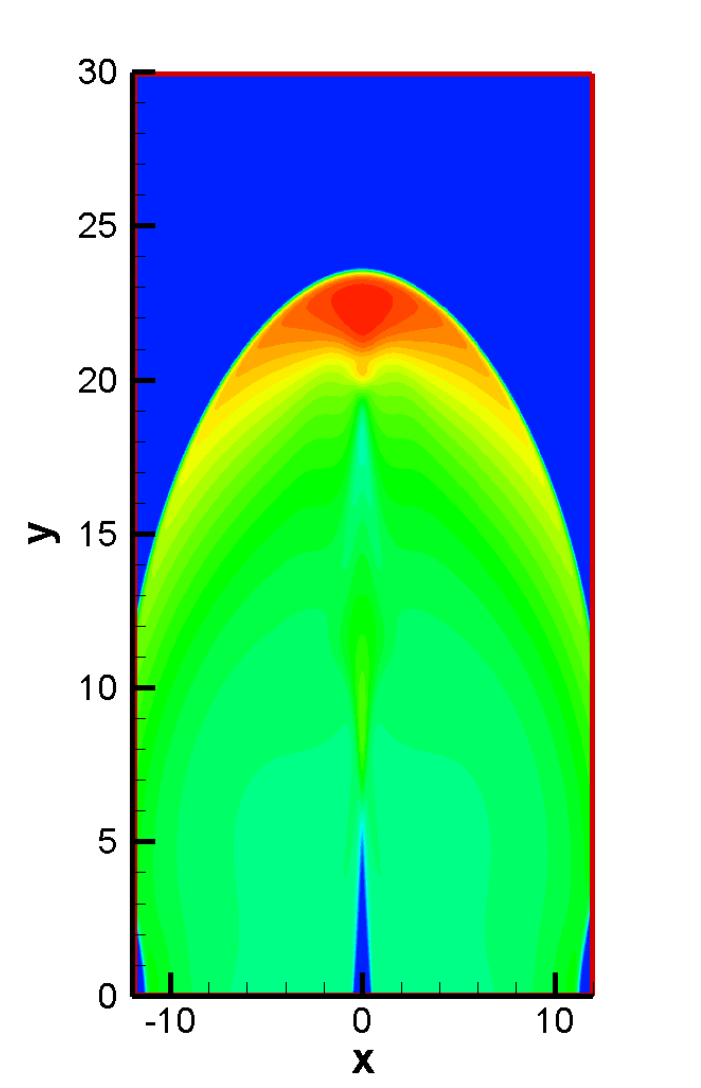}
\includegraphics[width=1.5in]{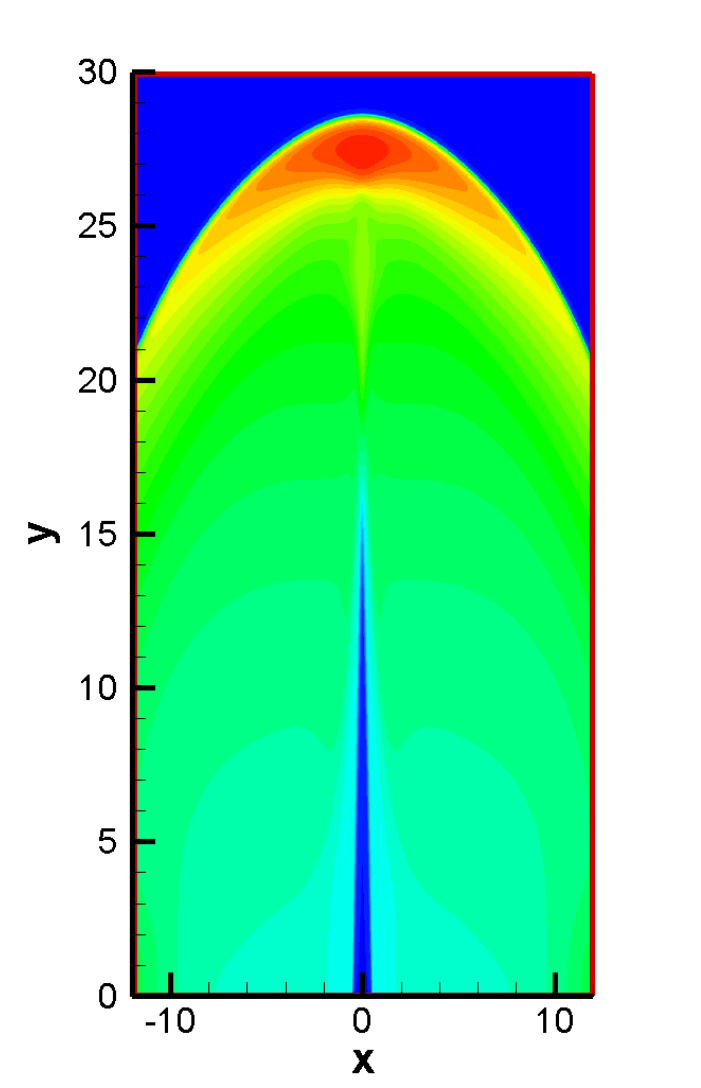}
\caption{\small Same as Figure \ref{fig4} except for the  pressure logarithm $\ln p$.}
\label{fig5}
\end{figure}

Next, we consider the pressure-matched highly supersonic
jet model, also called as
the cold model, in which the relativistic effects from the large
beam speed dominate, so that there exists an important
difference between the hot and cold relativistic jets.
The setups
are the same as the above hot jet model, except for that the rest-mass
density of the inlet jet becomes 0.1. The computational domain is
taken as
$(0,12)\times(0, 25)$, and again three different configurations
are considered here:
\begin{enumerate}[(\romannumeral1)]
\item $v_b=0.99$ and $M_b=50$, corresponding to the case of $\gamma\approx7.088$ and
$M_r\approx354.371$.
\item $v_b=0.999$ and $M_b=50$, corresponding to the case of $\gamma\approx22.366$ and
$M_r\approx1118.090$.
\item $v_b=0.9999$ and $M_b=500$, corresponding to the case of $\gamma\approx70.712$ and
$M_r\approx35356.152$.
\end{enumerate}
Figures \ref{fig6} and \ref{fig7} show the schlieren images of the
rest-mass density logarithm $\ln\rho$ and the
pressure logarithm $\ln p$ within the domain $[-12,12]\times[0,25]$, obtained by using
the first order multidimensional HLL scheme with $240\times500$ uniform cells for the computational domain $(0,12)\times(0,25)$.
\end{example}
\begin{figure}[H]
\centering
\includegraphics[width=1.5in]{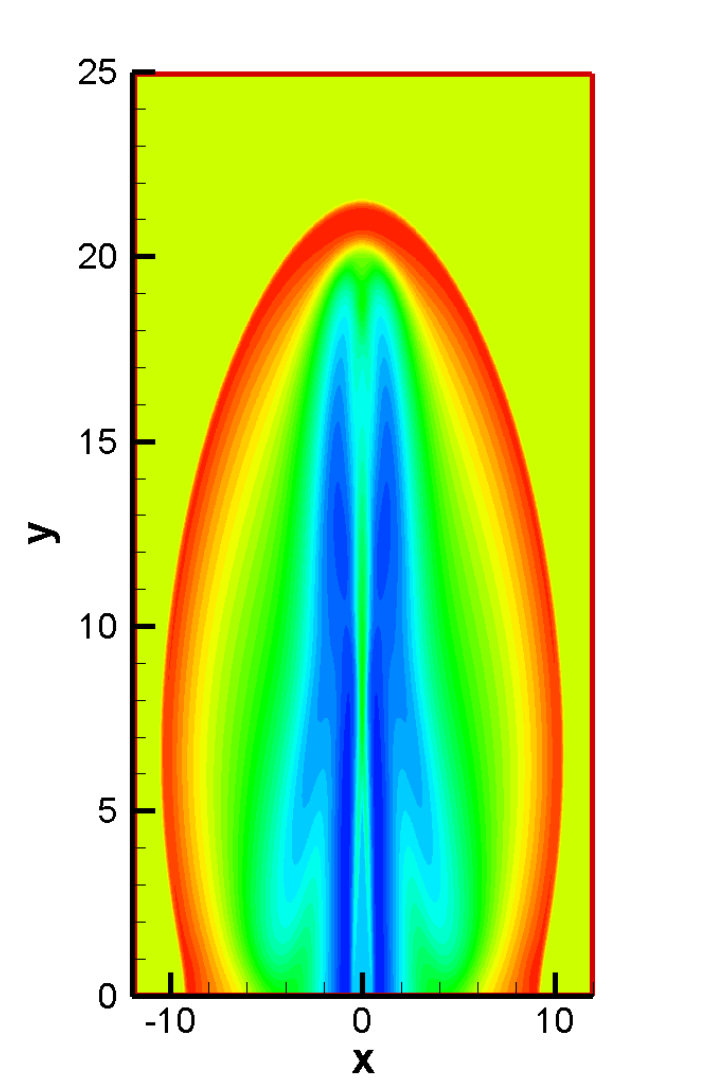}
\includegraphics[width=1.5in]{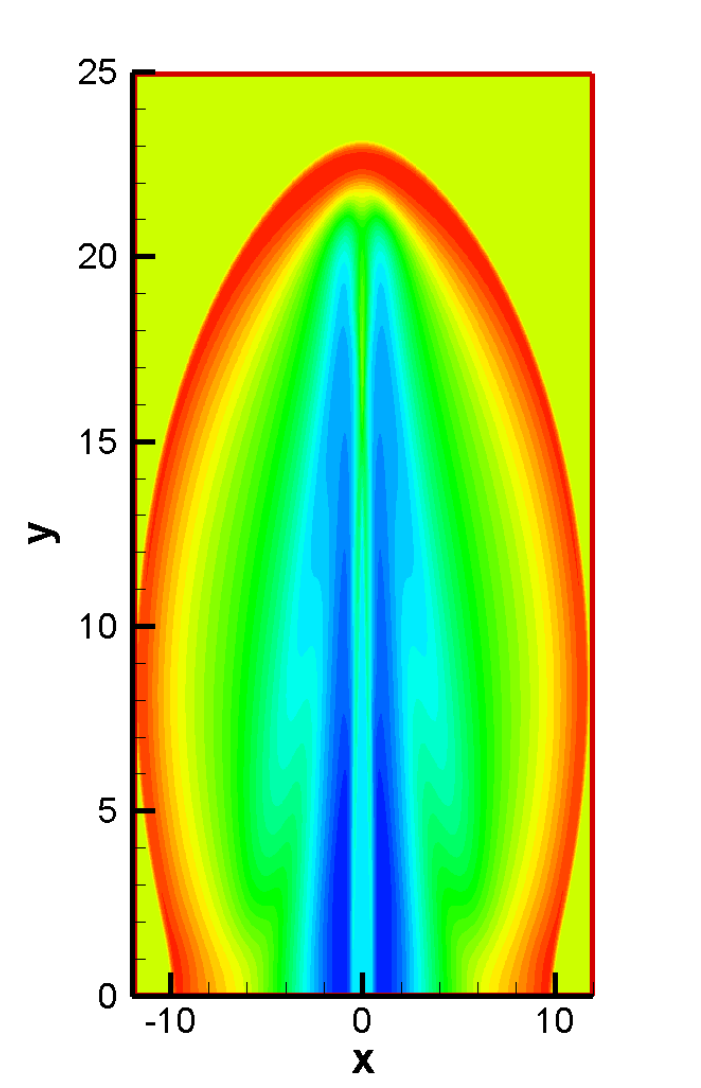}
\includegraphics[width=1.5in]{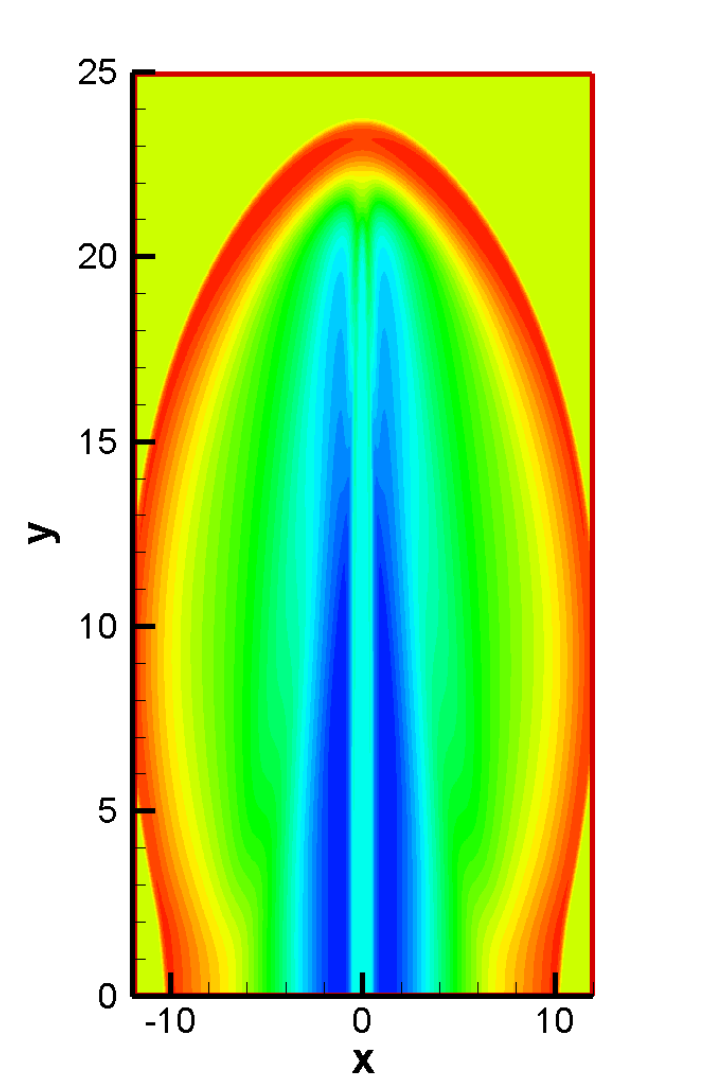}
\caption{\small Example \ref{exam6}: Schlieren images of rest-mass density logarithm $\ln\rho$ for the  cold jet model obtained by
the first order scheme with $240\times500$ uniform cells. From left to right: configurations
(\romannumeral1) at $t=30$, (\romannumeral2) at $t=25$, and (\romannumeral3) at $t=23$.}
\label{fig6}
\end{figure}

\begin{figure}[H]
\centering
\includegraphics[width=1.5in]{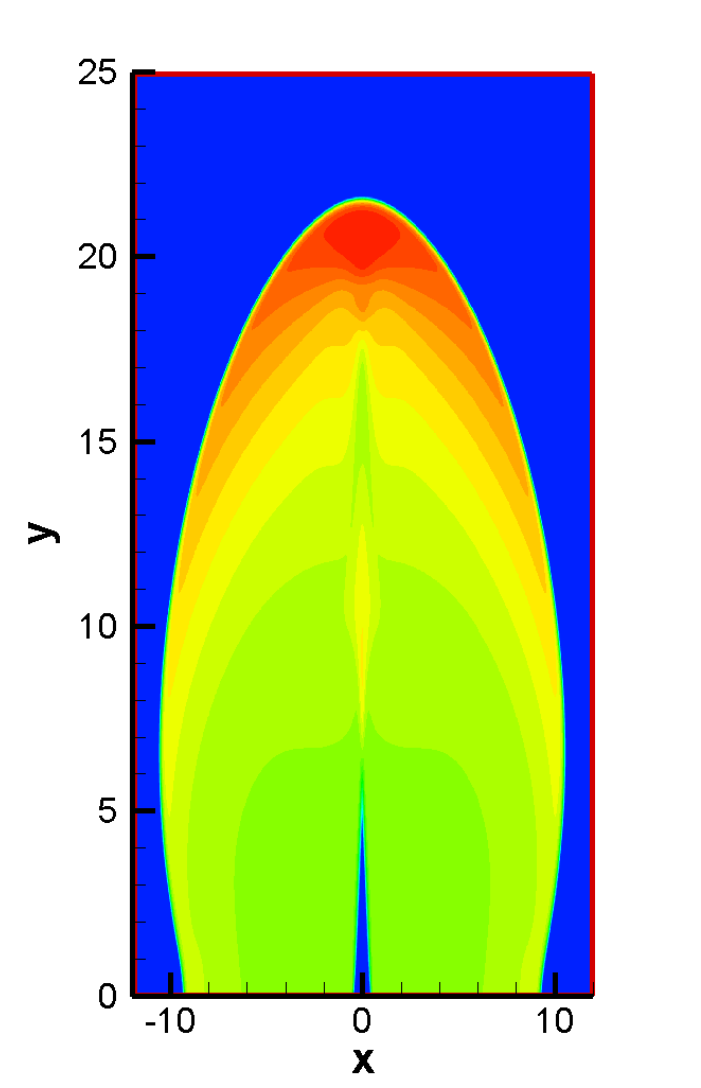}
\includegraphics[width=1.5in]{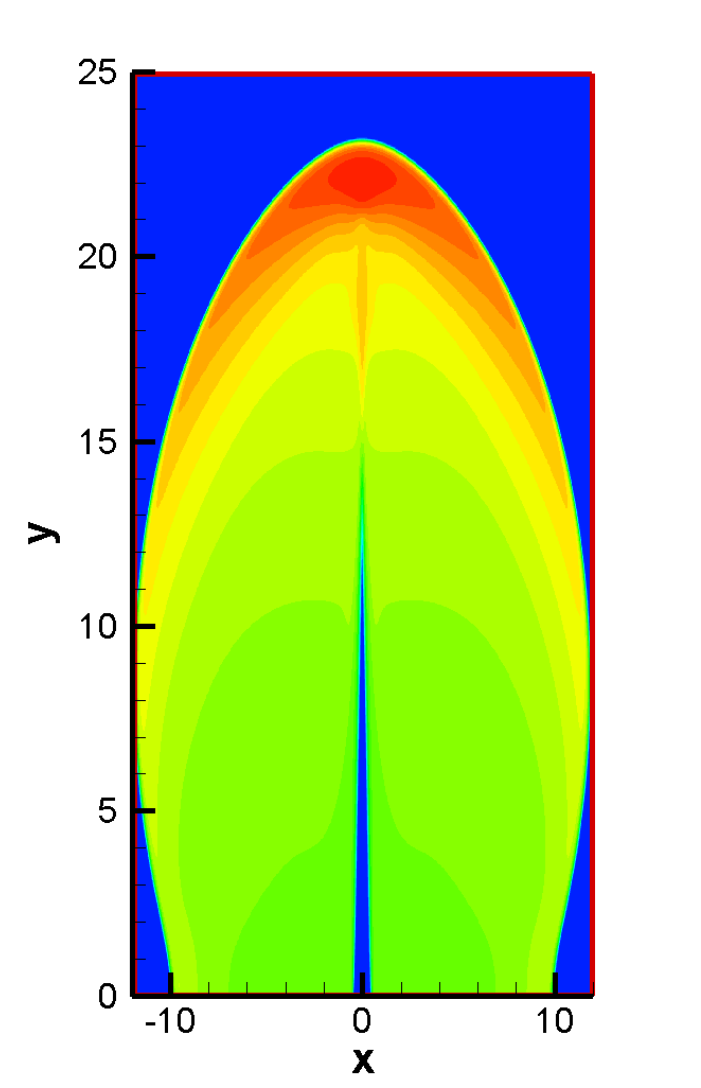}
\includegraphics[width=1.5in]{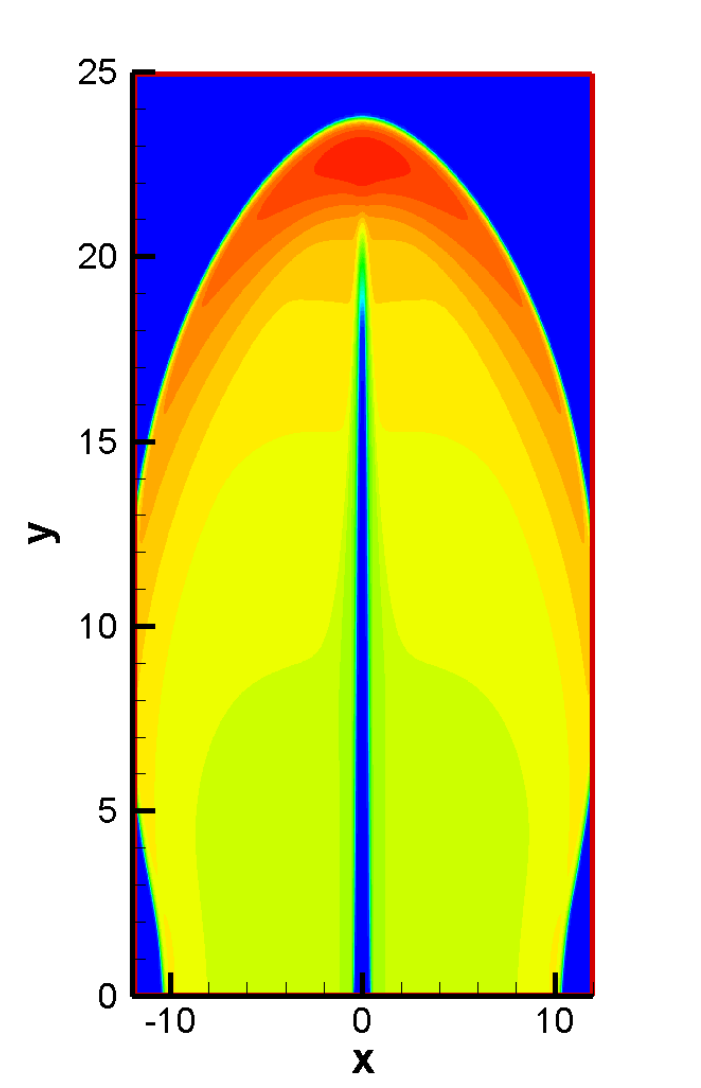}
\caption{\small Same as Figure \ref{fig6} except for the pressure logarithm $\ln p$.}
\label{fig7}
\end{figure}

\section{Conclusion}\label{con}

This paper proposed a finite volume scheme based on the multidimensional HLL Riemann solver for the 2D special relativistic hydrodynamics and then studied its PCP property (i.e., preserving the positivity of the rest-mass density and the pressure  and the boundness of the fluid velocity).
We first proved that the intermediate states in the 1D HLL Riemann solver were PCP when the HLL wave speeds were estimated suitably, and  further obtain the PCP property of the intermediate states obtained from
the multidimensional HLL Riemann solver in a similar way.
Then we developed the first-order accurate PCP finite volume
scheme with the multidimensional HLL Riemann solver and forward Euler time discretization.
Finally, several 2D numerical experiments were conducted to
demonstrate the accuracy and the effectiveness of the proposed PCP scheme in solving
the special RHD problems involving large Lorentz factor, or low
rest-mass density or low pressure or strong discontinuities, etc.
It is worth to remark that the further development of the high-order accurate PCP genuinely multidimensional  scheme to higher-order of accuracy seems non-trivial, since
some different coefficients have been involved in the numerical fluxes, see \eqref{add23}
and \eqref{add24}, so that 
it is too difficult to rewrite a high-order
scheme as a convex combination of several first-order like schemes.
One possible way to overcome such difficulty is to design a PCP flux-limiter and then modify the numerical flux
to achieve the PCP property of the high-order scheme, see e.g. \cite{xu}.

\begin{appendix}
\renewcommand{\thesection}{Appendix \Alph{section}}
\section{Proof of Lemma \ref{lem2}}\label{appendix}
\renewcommand{\thesection}{\Alph{section}}
This appendix proves Lemma \ref{lem2} presented in
Section \ref{intro} in a slightly different way from
\cite{wu2015}.

(\romannumeral1) For any positive number $\kappa$, let $(D^\kappa,\bm{m}^\kappa,E^\kappa )^T=\bm{U}^{\kappa}:=\kappa\bm{U}$. Since $\bm{U}\in\mathcal{G}$, it is easy to verify
\begin{equation*}
\begin{aligned}
&D^{\kappa}=\kappa D>0,\ \
&E^\kappa-\sqrt{(D^\kappa)^2+|\bm{m}^\kappa|^2}=\kappa\big(E-\sqrt{D^2+|\bm{m}|^2}\big)>0,
\end{aligned}
\end{equation*}
which leads to admissibility of $\kappa\bm{U}$.

(\romannumeral2) The convexity of $\mathcal{G}$ shows
$$\frac{a_1}{a_1+a_2}\bm{U}_1+\frac{a_2}{a_1+a_2}\bm{U}_2\in\mathcal{G},$$
for any $a_1,a_2>0$ and $\bm{U}_1,\bm{U}_2\in\mathcal{G}$.
Combining it with the conclusion in
(\romannumeral1) yields
$$a_1\bm{U}_1+a_2\bm{U}_2\in\mathcal{G}.$$

(\romannumeral3) For simplicity, denote
$$\begin{aligned}
(D^\alpha,\bm{m}_i^\alpha,E^\alpha)^T=\bm{U}^\alpha:
&=\alpha\bm{U}-\bm{F}_i(\bm{U}),\\
(D^\beta,\bm{m}_i^\beta,E^\beta)^T=\bm{U}^\beta:
&=-\beta\bm{U}+\bm{F}_i(\bm{U}).
\end{aligned}$$
For the state $\bm{U}^\alpha$ with $\alpha\ge\lambda_i^{(4)}(\bm{U})$,
we can get
\begin{align*}
D^\alpha&=D(\alpha-u_i)\ge D\big(\lambda_i^{(4)}(\bm{U})-u_i\big)>0,\\
E^\alpha&=E(\alpha-u_i)-pu_i\ge E\big(\lambda_i^{(4)}(\bm{U})-u_i\big)-pu_i\\
&=\frac{p\gamma^{2}}{c_s^2}\bigg(\big(\Gamma-c_s^2\gamma^{-2}\big)
\frac{u_i(1-c_s^2)+c_s\gamma^{-1}\sqrt{1-u_i^2-c_s^2(|\bm{u}|^2-u_i^2)}}
{1-c_s^2|\bm{u}|^2}-\Gamma u_i\bigg)\\
&\ge\frac{p\gamma^{2}}{c_s^2}\bigg(\big(\Gamma-c_s^2\gamma^{-2}\big)
\frac{u_i(1-c_s^{2})+c_s\gamma^{-2}}{1-c_s^{2}|\bm{u}|^2}-\Gamma u_i\bigg)\\
&=\frac{p}{c_s(1-c_s^{2}|\bm{u}|^2)}\bigg(-c_su_i(\Gamma-c_s^2+1)
+\Gamma-c_s^2\gamma^{-2}\bigg)\\
&\ge\frac{p}{c_s(1-c_s^{2}|\bm{u}|^2)}\bigg(-c_s|
\bm{u}|(\Gamma-c_s^2+1)+\Gamma-c_s^2\gamma^{-2}\bigg)\\
&=\frac{p}{c_s(1+c_s|\bm{u}|)}\bigg(\Gamma-c_s^2-c_s|\bm{u}|\bigg)>0,
\end{align*}
and
\begin{align*}
(E^\alpha)^2-|\bm{m}^\alpha|^2-(D^\alpha)^2&=(E^2-|\bm{m}|^2-D^2-p^2)
(\alpha-u_i)^2+p^2(\alpha^2-1)\\
&=\frac{\Gamma p^2}{c_s^2}\gamma^{2}\bigg(\frac{2}{\Gamma-1}-\frac{\Gamma c_s^2}{(\Gamma-1)^2}\bigg)
(\alpha-u_i)^2+p^2(\alpha^2-1)\\
&=p^2\cdot f(\alpha),
\end{align*}
where $f(s)$ is a quadratic function of $s\in[\lambda_i^{(4)}(\bm{U}),1)$ with the form of
\begin{equation*}
f(s)=\frac{\Gamma\gamma^{2}}{c_s^2}\bigg(\frac{2}{\Gamma-1}-\frac{\Gamma c_s^2}{(\Gamma-1)^2}\bigg)
(s-u_i)^2+s^2-1.
\end{equation*}
It is easy to prove that $f(s)$ is monotonically increasing with $s\in[\lambda_i^{(4)}(\bm{U}),1)$,
so that
$f(s)\ge f(\lambda_i^{(4)}(\bm{U}))$ for any $s\in[\lambda_i^{(4)}(\bm{U}),1)$ and then
$f(\alpha)\ge f(\lambda_i^{(4)}(\bm{U}))$. Moreover, we   have
\begin{equation*}
\begin{aligned}
f(\lambda_i^{(4)}(\bm{U}))&=\frac{2\Gamma(\Gamma-1)-\Gamma^2c_s^2}{(\Gamma-1)^2(1-c_s^2|\bm{u}|^2)^2}
\bigg(-\frac{c_su_i}{\gamma}+\sqrt{1-u_i^{2}-c_s^{2}(|\bm{u}|^2-u_i^2)}\bigg)^2\\
&~~~+\frac{\bigg(u_i(1-c_s^2)+c_s\gamma^{-1}\sqrt{1-u_i^{2}-c_s^{2}(|\bm{u}|^2-u_i^2)}\bigg)^2}
{(1-c_s^2|\bm{u}|^2)^2}-1\\
&=C_1\bigg(c_su_i\gamma^{-1}-\sqrt{1-u_i^{2}-c_s^{2}(|\bm{u}|^2-u_i^2)}\bigg)^2\ge0,
\end{aligned}
\end{equation*}
with
$$C_1=\frac{1}{(1-c_s^2|\bm{u}|^2)^2}\bigg(\Gamma^2-1+c_s^2(1-2\Gamma)\bigg)>0.$$
Therefore,   $f(\alpha)>0$ and then $(E^\alpha)^2-|\bm{m}^\alpha|^2-(D^\alpha)^2>0$. So far,
we have proved the conclusion $\alpha\bm{U}-\bm{F}(\bm{U})\in \mathcal{G}$ for $\alpha\ge\lambda_i^{(d+2)}(\bm{U})$.

For the state  $\bm{U}^\beta$ with $\beta\le\lambda_i^{(1)}(\bm{U})$,
one can similarly have
\begin{align*}
D^\beta&=D(u_i-\beta)\ge D\big(u_i-\lambda_i^{(1)}(\bm{U})\big)>0,\\
E^\beta&=E(u_i-\beta)+pu_i\ge E\big(u_i-\lambda_i^{(1)}(\bm{U})\big)+pu_i\\
&=\frac{p\gamma^{2}}{c_s^2}\bigg(-\big(\Gamma-c_s^2\gamma^{-2}\big)
\frac{u_i(1-c_s^{2})-c_s\gamma^{-1}\sqrt{1-u_i^2-c_s^2(|\bm{u}|^2-u_i^2)}}
{1-c_s^{2}|\bm{u}|^2}+\Gamma u_i\bigg)\\
&\ge\frac{p\gamma^{2}}{c_s^2}\bigg(-\big(\Gamma-c_s^2\gamma^{-2}\big)
\frac{u_i(1-c_s^{2})-c_s\gamma^{-2}}{1-c_s^{2}|\bm{u}|^2}+\Gamma u_i\bigg)\\
&=\frac{p}{c_s(1-c_s^{2}|\bm{u}|^2)}\bigg(c_su_i(\Gamma-c_s^2+1)+\Gamma-c_s^2\gamma^{-2}\bigg)\\
&\ge\frac{p}{c_s(1-c_s^{2}|\bm{u}|^2)}\bigg(-c_s|\bm{u}|(\Gamma-c_s^2+1)+\Gamma-c_s^2\gamma^{-2}\bigg)\\
&=\frac{p}{c_s(1+c_s|\bm{u}|)}\bigg(\Gamma-c_s^2-c_s|\bm{u}|\bigg)>0,
\end{align*}
and
\begin{align*}
(E^\beta)^2-|\bm{m}^\beta|^2-(D^\beta)^2&=(E^2-|\bm{m}|^2-D^2-p^2)(\beta-u_i)^2+p^2(\beta^2-1)\\
&=\frac{\Gamma p^2}{c_s^2}\gamma^{2}\bigg(\frac{2}{\Gamma-1}-\frac{\Gamma c_s^2}{(\Gamma-1)^2}\bigg)
(\beta-u_i)^2+p^2(\beta^2-1)\\
&=p^2\cdot g(\beta),
\end{align*}
where $g(s)$ is a quadratic function of $s\in(-1,\lambda_i^{(1)}(\bm{U})]$ with the form of
\begin{equation*}
g(s)=\frac{\Gamma\gamma^{2}}{c_s^2}\bigg(\frac{2}{\Gamma-1}-\frac{\Gamma c_s^2}{(\Gamma-1)^2}\bigg)
(s-u_i)^2+s^2-1.
\end{equation*}
It is easy to  prove that $g(s)$ is monotonically decreasing with $s\in(-1,\lambda_i^{(1)}(\bm{U})]$,
so that $g(s)\ge g(\lambda_i^{(1)}(\bm{U}))$ for any $s\in(-1,\lambda_i^{(1)}(\bm{U})]$ and then
$g(\beta)\ge g(\lambda_i^{(1)}(\bm{U}))$. Moreover, we can show
\begin{align*}
g(\lambda_i^{(1)}(\bm{U}))&=\frac{2\Gamma(\Gamma-1)-\Gamma^2c_s^2}{(\Gamma-1)^2(1-c_s^2|\bm{u}|^2)^2}
\bigg(c_su_i\gamma^{-1}+\sqrt{1-u^{2}-c_s^{2}(|\bm{u}|^2-u_i^2)}\bigg)^2\\
&~~~+\frac{\bigg(u_i(1-c_s^2)-c_s\gamma^{-1}\sqrt{1-u_i^{2}-c_s^{2}(|\bm{u}|^2-u_i^2)}\bigg)^2}
{(1-c_s^2|\bm{u}|^2)^2}-1\\
&=C_2\bigg(c_su_i\gamma^{-1}+\sqrt{1-u_i^{2}-c_s^{2}(|\bm{u}|^2-u_i^2)}\bigg)^2\ge0,
\end{align*}
with
$$C_2=\frac{1}{(1-c_s^2|\bm{u}|^2)^2}\bigg(\Gamma^2-1+c_s^2(1-2\Gamma)\bigg)>0.$$
Therefore,  $g(\beta)>0$ and then $(E^\beta)^2-|\bm{m}^\beta|^2-(D^\beta)^2>0$, which leads to
$-\beta\bm{U}+\bm{F}_i(\bm{U})\in \mathcal{G}$ for $\beta\le\lambda_i^{(1)}(\bm{U})$.
\qed
\end{appendix}

\end{document}